\documentclass{article}
\usepackage[a4paper]{geometry}
\usepackage[fleqn]{amsmath}
\usepackage{amssymb,amsthm,doi,enumitem,environ,float,color,mathrsfs,mleftright,subcaption,tikz,tikz-3dplot}
\usepackage[initials]{amsrefs}

\newtheorem{theorem}{Theorem}[section]
\newtheorem{lemma}[theorem]{Lemma}
\newtheorem{proposition}[theorem]{Proposition}

\newtheorem{question}[theorem]{Question}
\newtheorem{conjecture}[theorem]{Conjecture}
\theoremstyle{definition}
\newtheorem*{definition}{Definition}

\theoremstyle{plain}

\newtheorem{innerthm}{Theorem}
\newenvironment{maintheorem}[1]
  {\renewcommand\theinnerthm{#1}\innerthm}
  {\endinnerthm}
  
\newtheorem{innercor}{Corollary}
\newenvironment{maincorollary}[1]
  {\renewcommand\theinnercor{#1}\innercor}
  {\endinnercor}  

\graphicspath{ {figures/} }

\numberwithin{equation}{section}

\renewcommand{\geq}{\geqslant}
\renewcommand{\leq}{\leqslant}
\DeclareMathOperator{\Ex}{Ex}

\setlist[enumerate]{leftmargin=20pt,itemsep=0pt,topsep=0pt}
\setlist[enumerate,1]{label=\emph{(\roman*)},ref={(\roman*)}}
\setlist[enumerate,2]{label=\emph{(\alph*)},ref={(\alph*)}}

\renewcommand{\Re}{\operatorname{Re}}

\makeatletter
\renewcommand\section{\@startsection {section}{1}{\z@}%
                                   {-3.5ex \@plus -1ex \@minus -.2ex}%
                                   {1.3ex \@plus.2ex}%
                                   {\normalfont\bf\large}}
\makeatother

\makeatletter
\renewcommand\subsection{\@startsection {subsection}{1}{\z@}%
                                   {-3.5ex \@plus -1ex \@minus -.2ex}%
                                   {0.1ex \@plus.2ex}%
                                   {\normalfont\bf\normalsize}}
\makeatother

\usetikzlibrary{calc}
\usetikzlibrary{matrix}
\usetikzlibrary{decorations.pathreplacing,decorations.markings}
\usetikzlibrary{arrows.meta}
\usetikzlibrary{arrows,shapes,positioning}
\tikzstyle arrowstyle=[scale=1]

\tikzset{
	midarrow/.style={
		postaction={decorate,decoration={markings,mark=at position 0.5 with {\arrow{Stealth[length=#1,width=#1]}}}}
	}
}

\tikzset{pointer/.style 2 args={draw,fill,single arrow,
    single arrow tip angle=45,
    single arrow head extend=#1,
    single arrow head indent=0pt,
    inner sep=0pt,
    rotate=#2}}

\definecolor{gammacol}{RGB}{33,120,33}
\definecolor{deltacol}{RGB}{44,90,160}
\definecolor{alphacol}{RGB}{170,0,0}
\definecolor{betacol}{RGB}{171,55,200}

\tikzset{
  gridline/.style={black,line width=0.35pt},
  pathline/.style={black,line width=1.05pt},
  dot/.style={circle,fill=black,inner sep=1.55pt},
  rational/.style={black,font=\scriptsize,inner sep=1pt},
  gint/.style={blue!70!black,font=\scriptsize,inner sep=1pt},
  directed/.style={
    postaction={decorate},
    decoration={markings,mark=at position .55 with {\arrow{Stealth[length=2.0mm,width=1.35mm]}}}
  },
  directedtwo/.style={
    postaction={decorate},
    decoration={markings,
      mark=at position .33 with {\arrow{Stealth[length=2.0mm,width=1.35mm]}},
      mark=at position .67 with {\arrow{Stealth[length=2.0mm,width=1.35mm]}}}
  }
}

\newcommand{\Nlab}{N}
\newcommand{\VertexLabelFont}{\tiny}
\newcommand{\TriangulationLineWidth}{0.35pt}

\tikzset{
  triangulation/.style={line cap=round,line join=round,line width=\TriangulationLineWidth},
  edge/.style={triangulation},
  vlab/.style={font=\VertexLabelFont,inner sep=0.6pt,outer sep=0pt}
}

\newlength{\VertexLabelGap}
\newlength{\CrowdedVertexLabelGap}
\tikzset{
  triangulation/.style={
    line width=.35pt,
    line cap=round,
    line join=round
  },
  edge/.style={triangulation},
  vlabel/.style={
    font=\VertexLabelFont,
    inner sep=.35pt,
    outer sep=0pt
  }
}

\newcommand{\CoordList}[2]{%
  \foreach \i/\x/\y in {#2}{\coordinate (#1\i) at (\x,\y);}%
}

\newif\ifMirrorFigure
\tikzset{mirror/.code={\MirrorFiguretrue\pgfkeysalso{xscale=-1}}}
\newcommand{\VLabel}[4][\VertexLabelGap]{%
  \pgfmathsetmacro{\LabelAngle}{\ifMirrorFigure 180-(#3)\else #3\fi}%
  \node[
    inner sep=0pt,
    outer sep=0pt,
    label={[vlabel,label distance=#1]\LabelAngle:{$#4$}}
  ] at (#2) {};%
}

\newcommand{\VLabelList}[2]{%
  \foreach \i/\angle/\text in {#2}{\VLabel{#1\i}{\angle}{\text}}%
}

\newcommand{\PairedLabel}[5][\VertexLabelGap]{%
  \VLabel[#1]{#2}{#4}{#5}%
  \pgfmathsetmacro{\OppositeLabelAngle}{#4+180}%
  \VLabel[#1]{#3}{\OppositeLabelAngle}{#5}%
}

\newcommand{\ReflectChain}[3]{%
  \foreach \i in {0,...,#3}{%
    \coordinate (#2\i) at ($(#1\i)!2!(0,0)$);%
  }%
}

\newcommand{\DrawFan}[3]{%
  \pgfmathtruncatemacro{\lastedge}{#3-1}%
  \draw[edge] (#1)--(#20);
  \foreach \i [evaluate=\i as \j using int(\i+1)] in {0,...,\lastedge}{%
    \draw[edge] (#2\i)--(#2\j) (#1)--(#2\j);%
  }%
}

\newcommand{\DrawFanWithEars}[4]{%
  \pgfmathtruncatemacro{\lastedge}{#4-1}%
  \draw[edge] (#1)--(#20);
  \foreach \i [evaluate=\i as \j using int(\i+1)] in {0,...,\lastedge}{%
    \draw[edge]
      (#2\i)--(#2\j)
      (#1)--(#2\j)
      (#2\i)--(#3\i)--(#2\j);%
  }%
}

\newcommand{\DrawFanWithDoubleEars}[5]{%
  \pgfmathtruncatemacro{\lastedge}{#5-1}%
  \draw[edge] (#1)--(#20);
  \foreach \i [evaluate=\i as \j using int(\i+1)] in {0,...,\lastedge}{%
    \draw[edge]
      (#2\i)--(#2\j)
      (#1)--(#2\j)
      (#2\i)--(#3\i)--(#2\j)
      (#2\i)--(#4\i)--(#3\i);%
  }%
}

\newcommand{\PolySub}[3][.31\linewidth]{%
  \begin{subfigure}[t]{#1}
    \centering
    #2
    \caption{#3}
  \end{subfigure}%
}

\newcommand{\PolyA}{%
\begin{tikzpicture}[triangulation]
  \foreach \i/\angle in {0/210,1/330,2/90}{%
    \coordinate (a\i) at (\angle:.55);
    \VLabel{a\i}{\angle}{1}
  }
  \draw[edge] (a0)--(a1)--(a2)--cycle;
\end{tikzpicture}%
}

\newcommand{\PolyB}{%
\begin{tikzpicture}[triangulation]
  \foreach \i/\angle/\text in {0/225/1,1/315/2,2/45/1,3/135/2}{%
    \coordinate (b\i) at (\angle:.7778);
    \VLabel{b\i}{\angle}{\text}
  }
  \draw[edge] (b0)--(b1)--(b2)--(b3)--cycle (b3)--(b1);
\end{tikzpicture}%
}

\newcommand{\Hexagon}[2]{%
  \foreach \i in {0,...,5}{\coordinate (#1\i) at ({180-60*\i}:1);}%
  \draw[edge] (#10)--(#11)--(#12)--(#13)--(#14)--(#15)--cycle;
  \foreach \i/\text in {#2}{\VLabel{#1\i}{180-60*\i}{\text}}%
}

\newcommand{\PolyC}{%
\begin{tikzpicture}[triangulation]
  \Hexagon{c}{0/3,1/1,2/3,3/1,4/3,5/1}
  \draw[edge] (c0)--(c2) (c0)--(c4) (c2)--(c4);
\end{tikzpicture}%
}

\newcommand{\PolyD}{%
\begin{tikzpicture}[triangulation,mirror]
  \Hexagon{d}{0/1,1/2,2/3,3/1,4/2,5/3}
  \draw[edge] (d1)--(d5) (d2)--(d5) (d2)--(d4);
\end{tikzpicture}%
}

\newcommand{\PolyE}{%
\begin{tikzpicture}[triangulation,mirror]
  \CoordList{e}{%
    0/-.75/-1.7321, 1/-.25/-.8660, 2/.75/-.8660,
    3/1.25/0,      4/.25/0,        5/-.25/.8660,
    6/-1.25/.8660, 7/-.75/0,       8/-1.25/-.8660}
  \draw[edge]
    (e0)--(e1)--(e2)--(e3)--(e4)--(e5)--(e6)--(e7)--(e8)--cycle
    (e1)--(e8) (e7)--(e1) (e7)--(e4) (e7)--(e5)
    (e1)--(e4) (e2)--(e4);
  \VLabelList{e}{%
    0/270/1, 1/300/4, 2/300/2,
    3/30/1,  4/60/4,  5/60/2,
    6/150/1, 7/180/4, 8/180/2}
\end{tikzpicture}%
}

\newcommand{\PolyF}{%
\begin{tikzpicture}[triangulation,mirror]
  \CoordList{f}{%
    0/1.75/-1.7321, 1/.75/-1.7321, 2/.25/-.8660,
    3/-.75/-.8660, 4/-1.75/-.8660,5/-1.25/0,
    6/-.25/0,      7/.25/.8660,   8/.75/1.7321,
    9/1.25/.8660,  10/.75/0,      11/1.25/-.8660}
  \draw[edge]
    (f0)--(f1)--(f2)--(f3)--(f4)--(f5)--(f6)--(f7)--(f8)--(f9)--(f10)--(f11)--cycle
    (f1)--(f11) (f2)--(f11) (f10)--(f2) (f10)--(f6)
    (f10)--(f7) (f7)--(f9) (f2)--(f6) (f3)--(f5) (f3)--(f6);
  \VLabelList{f}{%
    0/330/1, 1/240/2, 2/260/4, 3/270/3,
    4/210/1, 5/120/2, 6/100/4, 7/150/3,
    8/90/1,  9/0/2,   10/0/4,  11/30/3}
\end{tikzpicture}%
}

\newcommand{\PolyG}{%
\begin{tikzpicture}[triangulation,scale=.54,mirror]
  \coordinate (gA) at (-1.5,-.7);
  \coordinate (gB) at ( 1.5, .7);
  \CoordList{gP}{%
    0/-4.036/1.428, 1/-3.003/2.250, 2/-1.731/2.602,
    3/-.423/2.430,  4/.715/1.761}
  \coordinate (gP5) at (gB);
  \ReflectChain{gP}{gQ}{5}

  \DrawFan{gA}{gP}{5}
  \DrawFan{gB}{gQ}{5}

  \PairedLabel[\CrowdedVertexLabelGap]{gA}{gB}{223.25}{\Nlab}
  \foreach \i/\angle/\text in {0/180/1,1/117/2,2/94/2,3/71/2,4/48/2}{%
    \PairedLabel{gP\i}{gQ\i}{\angle}{\text}
  }
\end{tikzpicture}%
}

\newcommand{\PolyH}{%
\begin{tikzpicture}[triangulation,mirror]
  \coordinate (hA) at (.62,-.372);
  \coordinate (hB) at (-.62,.372);
  \CoordList{hP}{%
    0/1.961/.170, 1/1.599/.693, 2/1.047/1.010,
    3/.413/1.059, 4/-.181/.832}
  \coordinate (hP5) at (hB);
  \CoordList{hT}{%
    0/2.233/.745, 1/1.597/1.329, 2/.773/1.583,
    3/-.081/1.460,4/-.799/.982}
  \ReflectChain{hP}{hQ}{5}
  \ReflectChain{hT}{hU}{4}

  \DrawFanWithEars{hA}{hP}{hT}{5}
  \DrawFanWithEars{hB}{hQ}{hU}{5}

  \PairedLabel[\CrowdedVertexLabelGap]{hA}{hB}{334.18}{\Nlab}
  \PairedLabel{hP0}{hQ0}{313.35}{2}
  \foreach \i [evaluate=\i as \angle using 22.04+25.40*\i] in {1,4}{%
    \PairedLabel[0.7mm]{hP\i}{hQ\i}{\angle}{4}
  }
  \foreach \i [evaluate=\i as \angle using 25+21*\i] in {2,3}{%
    \PairedLabel[0.7mm]{hP\i}{hQ\i}{\angle}{4}
  }
   \foreach \i [evaluate=\i as \angle using 34.69+25.40*\i] in {0,...,4}{%
    \PairedLabel{hT\i}{hU\i}{\angle}{1}
  }
\end{tikzpicture}%
}

\newcommand{\PolyI}{%
\begin{tikzpicture}[triangulation,mirror]
  \coordinate (iA) at (-.65,-.325);
  \coordinate (iB) at ( .65, .325);

  \def\R{1.453444185}
  \def\Step{25}
  \def\ThetaB{26.565051177}
  \def\EarHeight{.5}
  \def\TipHeight{.23}

  \pgfmathsetmacro{\sx}{(\R+\EarHeight)*cos(\Step/2)}
  \pgfmathsetmacro{\sy}{-(\R+\EarHeight)*sin(\Step/2)}
  \pgfmathsetmacro{\dx}{\sx-\R}
  \pgfmathsetmacro{\dy}{\sy}
  \pgfmathsetmacro{\baselen}{veclen(\dx,\dy)}
  \pgfmathsetmacro{\tx}{(\R+\sx)/2-\TipHeight*\dy/\baselen}
  \pgfmathsetmacro{\ty}{\sy/2+\TipHeight*\dx/\baselen}

  \foreach \j [evaluate=\j as \angle using \ThetaB+(5-\j)*\Step] in {0,...,4}{%
    \begin{scope}[shift={(iA)},rotate=\angle]
      \coordinate (iP\j) at (\R,0);
      \coordinate (iS\j) at (\sx,\sy);
      \coordinate (iT\j) at (\tx,\ty);
    \end{scope}
  }
  \coordinate (iP5) at (iB);

  \ReflectChain{iP}{iQ}{5}
  \ReflectChain{iS}{iU}{4}
  \ReflectChain{iT}{iV}{4}

  \DrawFanWithDoubleEars{iA}{iP}{iS}{iT}{5}
  \DrawFanWithDoubleEars{iB}{iQ}{iU}{iV}{5}

  \PairedLabel[\CrowdedVertexLabelGap]{iA}{iB}{204.96}{\Nlab}
  \PairedLabel{iP0}{iQ0}{237.07}{3}

  \foreach \i [evaluate=\i as \angle using 198-45*\i] in {1,...,2}{%
    \PairedLabel[3.5*\CrowdedVertexLabelGap]{iP\i}{iQ\i}{\angle}{5}
  }
  \foreach \i [evaluate=\i as \angle using 107-5*\i] in {3,...,4}{%
    \PairedLabel[3.5*\CrowdedVertexLabelGap]{iP\i}{iQ\i}{\angle}{5}
  }  
  
  \foreach \i [evaluate=\i as \angle using 189.773-25*\i] in {0,...,4}{%
    \PairedLabel{iT\i}{iV\i}{\angle}{1}
  }
  \foreach \i [evaluate=\i as \angle using 117.665-25*\i] in {0,...,4}{%
    \PairedLabel{iS\i}{iU\i}{\angle}{2}
  }
\end{tikzpicture}%
}

\newcommand{\GlobalScale}{0.90}
\pgfmathsetmacro{\ScaleOneOne}{2.55*\GlobalScale}
\pgfmathsetmacro{\ScaleTwoOne}{2.1*\GlobalScale}
\pgfmathsetmacro{\ScaleTwoTwo}{1.72*\GlobalScale}
\pgfmathsetmacro{\ScaleThreeThree}{1.32*\GlobalScale}

\tikzset{
  gridline/.style={black,line width=0.35pt},
  pathline/.style={black,line width=1.05pt,line cap=round,line join=round},
  dot/.style={circle,fill=black,inner sep=1.55pt},
  rational/.style={font=\scriptsize,inner sep=1pt},
  gint/.style={blue!70!black,font=\tiny,inner sep=1pt}
}

\tikzset{
  mid arrow/.style={
    postaction={decorate},
    decoration={
      markings,
      mark=at position 0.6 with {
        \arrow{Stealth[length=2.0mm,width=1.35mm]}
      }
    }
  }
}

\newcommand{\Edge}[2]{\draw[pathline, mid arrow] #1 -- #2;}
\newcommand{\BendEdge}[2]{\draw[pathline, mid arrow] #1 to[bend left=20] #2;}
\newcommand{\Spot}[1]{\node[dot] at #1 {};}
\newcommand{\GridOneOne}{\draw[gridline] (0,0) rectangle (1,1);}

\newcommand{\GridTwoTwo}{\draw[gridline] (-1,-1) rectangle (1,1);\draw[gridline] (-1,0)--(1,0);\draw[gridline] (0,-1)--(0,1);}
\newcommand{\GridThreeThree}{\draw[gridline] (-2,-2) rectangle (1,1);\draw[gridline] (-2,0)--(1,0);\draw[gridline] (-2,-1)--(1,-1);\draw[gridline] (0,-2)--(0,1);\draw[gridline] (-1,-2)--(-1,1);}
\newcommand{\BBoxOneOne}{\path[use as bounding box] (-0.40,-0.34) rectangle (1.40,1.34);}

\newcommand{\BBoxTwoTwo}{\path[use as bounding box] (-1.72,-1.42) rectangle (1.72,1.42);}

\newcommand{\MFig}[2]{%
  \begin{minipage}[t]{0.32\textwidth}\centering
    #1\\[-8pt]{\scriptsize #2}
  \end{minipage}%
}

\newcommand{\MFigWide}[2]{%
  \begin{minipage}[t]{0.48\textwidth}\centering
    #1\\[-8pt]{\scriptsize #2}
  \end{minipage}%
}

\newcommand{\MFigSlim}[2]{%
  \begin{minipage}[t]{0.25\textwidth}\centering
    #1\\[-8pt]{\scriptsize #2}
  \end{minipage}%
}

\AtEndDocument{%
	\par
	\medskip
	\begin{tabular}{@{}l@{}}%
		{Ian Short, Margaret Stanier, Andrei Zabolotskii}\\
		{School of Mathematics and Statistics, The Open University,}\\
		{Milton Keynes, MK7 6AA, United Kingdom}\\
		\textit{E-mail addresses}: \texttt{ian.short@open.ac.uk, margaret.stanier1@open.ac.uk,}\\
		\texttt{andrei.zabolotskii@open.ac.uk}\\	
		\\
		{Matty van Son}\\
		{Department of Computer Science, Mathematics \& Physics, University of the West Indies,}\\
		{Cave Hill Campus, P.O. Box 64, Bridgetown BB11000, Barbados}\\
		\textit{E-mail address}: \texttt{matty.vanson@uwi.edu}
\end{tabular}}

\title{ \vspace{-5ex}\bf \Large Optimal local convergence criteria for integer and Gaussian integer continued fractions}

\makeatletter
\renewcommand*{\@fnsymbol}[1]{\hspace*{-10pt}}
\makeatother

\author{Ian Short, Margaret Stanier, Matty van Son, and Andrei Zabolotskii\thanks{2020 Mathematics Subject Classification: Primary 11A55, 40A15; Secondary 05E16, 11B57.}\thanks{Key words: continued fraction, exclusion set, Farey graph,  quiddity sequence, triangulated polygon.}}

\date{\vspace{-5ex}}

\begin{document}

\maketitle

\begin{abstract}
The objective of this work is to determine optimal local restrictions on the coefficients of integer and  Gaussian integer continued fractions that imply convergence. We identify all minimal restrictions involving words of length two in the integer case, and we identify all reversible minimal restrictions of length two in the Gaussian integer case. In the integer setting, our classification is equivalent to a classification of minimal unavoidable words of length two in Conway--Coxeter quiddity sequences. We also construct a canonical set of restrictions of infinite cardinality that is strictly stronger than every finite set of restrictions.
\end{abstract}

\section{Introduction}

It is well-known that if the coefficients $b_i$ of a complex continued fraction satisfy $|b_i|\geq 2$, then the continued fraction converges. As a consequence, an \emph{integer} continued fraction with only finitely many coefficients  $0$, $1$, or $-1$ must converge. This result can easily be improved; for instance, Katok and Ugarcovici \cite{KaUg2005}*{Lemma~1.1} proved that any integer continued fraction with only finitely many substrings equal to $0$, $(1,b)$, or $(-1,-b)$, for $b>0$, converges. We describe each of the sets of words $\{-1,0,1\}$ and $\{\pm (1,b):b>0\}\cup\{0\}$ as `exclusion sets' because any integer continued fraction that has only finitely many occurrences of substrings from one of these sets converges. Our objective is to obtain optimal exclusion sets, for integer and Gaussian integer continued fractions. 

Throughout we will consider \emph{negative} continued fractions
\[
[b_0,b_1,\dotsc]=b_0- \cfrac{1}{b_1-\cfrac{1}{b_2-\cfrac{1}{\raisebox{-1ex}{$b_3-\dotsb$}}}}\,,
\]
where the coefficients are integers or Gaussian integers. These are more convenient for our purposes than the more familiar positive continued fractions (of the form $b_0+1/(b_1+ 1/(b_2+\dotsb ))$), and, in any case, it is straightforward to convert a positive continued fraction to an equivalent negative continued fraction (and vice versa) by changing the sign of odd-index coefficients.

There is a rich literature on the convergence of continued fractions; for a modern source on this topic, see \cite{LoWa2008}. Much attention has been paid to real or complex continued fractions, and there is comparatively little specifically on integer and Gaussian integer continued fractions. Probably the strongest known exclusion set for \emph{integer} continued fractions comes from the aforementioned work of Katok and Ugarcovici \cite{KaUg2005}*{Lemma~1.1}. They comment that their result ``might exist in the vast literature on continued fractions'', and indeed it does, in a slightly disguised form, in the work of Tietze \cite{Ti1911}. Offering a different perspective, the first and second authors described an algorithm in \cite{ShSt2022} to determine whether or not an integer continued fraction converges; that work cannot be framed in terms of exclusion sets.  

For Gaussian integer continued fractions,  it can be shown (with some effort) that the set 
\begin{equation}\label{equation99}
\hspace*{-1.8cm}U=\{(a,b)\in(\mathbb{Z}[i]\setminus\{0\})^2: ab=1,2,3,\pm i, 1\pm i,2\pm i, 3\pm i,1\pm 2i, 2\pm 2i, 3\pm 2i\}\cup\{0\}
\end{equation}
 is an exclusion set by applying the parabola theorems of Scott and Wall \cite{ScWa1940} and Leighton and Thron \cite{LeTh1942}. We call this a \emph{reversible} exclusion set because $(a,b)\in U$ if and only if $(b,a)\in U$. In \cite{DaNo2014}*{Theorem~6.7}, Dani and Nogueira  described an (ordered) exclusion set, which for negative continued fractions has the form $V=A\cup B\cup \{0,\pm 1, \pm i\} $, where
\begin{align*}
A&=\{(a,b)\in\mathbb{Z}[i]^2: |a|>1,|b|=\sqrt{2},|a-\bar{b}|<2\},\\
B&= \{(a,b)\in\mathbb{Z}[i]^2: |a|=\sqrt{2} ,1<|b|\leq 2,|ab-1|<|b|+1\}.
\end{align*}
This is neither stronger nor weaker than the reversible exclusion set $U$, in the sense that there are continued fractions excluded by $U$ but not by $V$, and vice versa.

We now define exclusion sets over a ring $R$ formally, where $R$ is $\mathbb{Z}$ or $\mathbb{Z}[i]$. Let $X$ be a (possibly infinite) set of finite sequences in $R$. We write elements of $X$ as words $x=(x_1,x_2,\dots,x_n)$ (unless $n=1$ in which case we write the word as $x_1$ rather than $(x_1)$). We say that $x$ is a \emph{substring} of $[b_0,b_1,\dotsc]$ (or of some word $(b_0,b_1,\dots,b_m$)) if $x=(b_r,b_{r+1},\dots,b_{s})$, for integers $r$ and $s$ with $0\leq r\leq s$. The set $X$ is \emph{without redundancy} if no element of $X$ is a substring of another. The \emph{reversal} of $X$ is the set of words obtained by reversing each element of $X$.

Notice that, for any unit $u$, $[b_0,b_1,\dotsc]$ converges if and only if $[ub_0,u^{-1}b_1,ub_2,u^{-1}b_3,\dotsc]$ converges. This motivates the next definition. Consider the operation $X\to X^u$ that replaces each element $(x_1,x_2,\dots,x_n)$ of $X$ with $(ux_1,u^{-1}x_2,\dots,u^{(-1)^{n-1}}x_n)$. We define $\overline{X}$ to be the smallest set of words in $R$ containing $X$ and $0$ that is closed under complex conjugation and closed under all operations $X\to X^u$, for every unit $u$ in $R$. When $R=\mathbb{Z}$, we simply have $\overline{X}=X\cup(-X)\cup\{0\}$. When $R=\mathbb{Z}[i]$, and $X$ is finite, the closure $\overline{X}$ is up to eight times as large as $X$ (plus one). We say that $X$ is \emph{symmetric} if $X=\overline{X}$. For a symmetric set $X$, we denote by $\Ex(X)$ the set of all continued fractions $[b_0,b_1,\dotsc]$ over $R$ for which every tail $[b_n,b_{n+1},\dotsc]$ contains a substring from $X$. 

\begin{definition}
A symmetric set $X$ without redundancy is an \emph{exclusion set} over $R$ if $\Ex(X)$ contains all continued fractions over $R$ that diverge. A \emph{reversible exclusion set} is an exclusion set that is closed under reversal.
\end{definition}

Thus, if only finitely many substrings of an integer continued fraction belong to an exclusion set $X$, then that continued fraction converges. Informally speaking, the better the exclusion set, the fewer \emph{convergent} continued fractions it excludes. We seek optimal exclusion sets.

For example, $X=\{-1,0,1\}$ is an exclusion set (over the integers); $\Ex(X)$ comprises integer continued fractions with infinitely many coefficients $0$, $1$, or $-1$. Also, $Y=\{\pm (1,b):b>0\}\cup\{0\}$ is an  exclusion set, and $\Ex(Y)\subsetneq \Ex(X)$ because $Y$ excludes fewer convergent continued fractions than $X$.
 
We define the \emph{length} of an exclusion set $X$ to be the supremum of the lengths of sequences in $X$. The length could be $+\infty$. There is a partial order $\preceq$ on the collection of \emph{finite} exclusion sets for which $X\preceq Y$ if $\Ex(X) \subseteq \Ex(Y)$. This partial order is antisymmetric because, for distinct sets $X$ and $Y$, we can find an element $s$ of one set such that no element of the other set is a substring of $s$. Then the continued fraction $[1,s,2,s,3,s,\dotsc]$ is excluded by one of $X$ or $Y$ and not the other. A \emph{minimal} exclusion set of length $L$ is a minimal element in the poset of finite exclusion sets of length $L$. (Throughout, minimality is only considered with respect to the poset of finite exclusion sets.) There is a unique minimal  exclusion set over the integers of length 1, namely $X=\{-1,0,1\}$. Our first main result is a complete classification of all minimal exclusion sets over the integers of length two. 

\begin{maintheorem}{A}\label{theoremA}
There are exactly eighteen minimal exclusion sets over the integers of length two, namely $\overline{O_i}$ and their reversals, where
\begin{align*}
O_1 &= \{(1,1), (1,2),(2,1),(1,3)\}\\
O_2 &= \{(1,1), (1,2),(2,2),(1,3),(1,4)\}\\
O_3 &= \{(1,1),(1,2),(2,2),(1,3),(4,1),(4,2)\}\\
O_4 &=\{(1,1), (1,2),(2,2),(3,1),(3,2),(4,1),(4,2),(1,5)\}\\
O_5 &=\{(1,1), (1,2),(2,2),(3,1),(3,2),(4,1),(4,2),(5,2)\}\\
O_6 &=\{(1,1), (1,2),(2,2),(3,1),(3,2),(1,4),(4,2),(1,5)\}\\
O_7 &=\{(1,1), (1,2),(2,2),(3,1),(3,2),(1,4),(4,2),(5,2)\}\\
O_8 &=\{(1,1), (1,2),(2,2),(3,1),(3,2), (1,4),(3,4),(1,5)\}\\
O_9 &=\{(1,1), (1,2),(2,2),(3,1),(3,2), (1,4),(3,4),(5,2)\}.
\end{align*}
\end{maintheorem}

In full, 
\(
\overline{O_1}=\{0,\pm(1,1), \pm(1,2),\pm(2,1),\pm(1,3)\},
\)
and similarly for the others. 

Our second main result is on infinite exclusion sets; these are exclusion sets of infinite cardinality that still consist of finite words. For an integer continued fraction $[b_0,b_1,\dots,b_n]$ we define $V_0=\infty$ and $V_{k+1}=[b_0,b_1,\dots,b_k]$, for $k=0,1,\dots,n$; these are the convergents of the continued fraction. We define $\widehat{O}$ to be the collection of all words $(b_0,b_1,\dots,b_n)$ for which $V_{n+1}=0$ and $V_0,V_1,\dots,V_{n+1}$ are pairwise distinct. For example, $(1,1)$ belongs to $\widehat{O}$ but $(1,1,1,1,1)$ does not, because although $[1,1,1,1,1]=0$ we also have $[1,1]=0$, so there is a repeated convergent. We will show that $\widehat{O}$ is an exclusion set that excludes fewer convergent continued fractions than \emph{any} finite exclusion set over the integers.

\begin{maintheorem}{B}\label{theoremB}
The set $\widehat{O}$ is an exclusion set over the integers and $\Ex(\widehat{O})\subsetneq \Ex(O)$ for any finite exclusion set $O$ over the integers.
\end{maintheorem}

To prove Theorems~\ref{theoremA} and~\ref{theoremB}, we use the correspondence between integer continued fractions and paths in the Farey graph described in, for example, \cites{BeHoSh2012,Ha2022}. The Farey graph is the infinite graph in the (closed) upper half-plane with vertices at the rationals (and infinity) and with an edge between two reduced rationals $a/b$ and $c/d$ if $ad-bc=\pm1$. The edges are represented by hyperbolic lines. A finite or infinite integer continued fraction can be represented in the Farey graph by a path with initial vertex $\infty$; the vertices of the path are the convergents of the continued fraction. For example, in Figure~\ref{figure1} the continued fraction $[1,-2,-2,1]$ is represented by the path $\langle \infty, 1,\tfrac32, \tfrac53,\tfrac85\rangle$ from $\infty$ to the value $\tfrac85$ of the continued fraction, and the convergents are 
\[
\infty,\qquad [1]=1,\qquad [1,-2]=\tfrac32,\qquad [1,-2,-2]=\tfrac53,\qquad [1,-2,-2,1]=\tfrac85.
\]
Positive coefficients $b_i$  correspond to right turns in the path and negative coefficients correspond to left turns (for $i>0$).

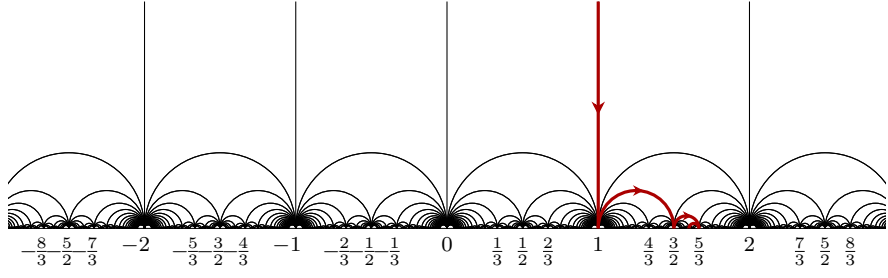
\begin{figure}[ht]
\centering
\begin{tikzpicture}[scale=2.0]
   \clip (-2.9,-0.3) rectangle (2.9,1.5);
   	
   \draw (-3,0) -- (3,0);	

   \foreach \leftend/\rightend in {0/1,0.0/0.5,0.5/1.0,0.0/0.333,0.333/0.5,0.0/0.25,0.25/0.333,0.0/0.2,0.2/0.25,0.0/0.167,0.167/0.2,0.0/0.143,0.143/0.167,0.0/0.125,0.125/0.143,0.0/0.111,0.111/0.125,0.0/0.1,0.1/0.111,0.0/0.091,0.091/0.1,0.0/0.083,0.083/0.091,0.0/0.077,0.077/0.083,0.0/0.071,0.071/0.077,0.0/0.067,0.067/0.071,0.0/0.062,0.062/0.067,0.0/0.059,0.059/0.062,0.0/0.056,0.056/0.059,0.0/0.053,0.053/0.056,0.0/0.05,0.05/0.053,0.0/0.048,0.048/0.05,0.05/0.051,0.051/0.053,0.053/0.054,0.054/0.056,0.056/0.057,0.057/0.059,0.059/0.061,0.061/0.062,0.062/0.065,0.065/0.067,0.067/0.069,0.069/0.071,0.071/0.074,0.074/0.077,0.077/0.08,0.08/0.083,0.083/0.087,0.087/0.091,0.091/0.095,0.095/0.1,0.1/0.105,0.105/0.111,0.111/0.118,0.118/0.125,0.125/0.133,0.133/0.143,0.143/0.154,0.154/0.167,0.167/0.182,0.182/0.2,0.2/0.222,0.222/0.25,0.25/0.286,0.286/0.333,0.25/0.273,0.273/0.286,0.286/0.3,0.3/0.333,0.333/0.4,0.4/0.5,0.333/0.375,0.375/0.4,0.333/0.364,0.364/0.375,0.375/0.385,0.385/0.4,0.4/0.429,0.429/0.5,0.4/0.417,0.417/0.429,0.429/0.444,0.444/0.5,0.429/0.438,0.438/0.444,0.444/0.455,0.455/0.5,0.444/0.45,0.45/0.455,0.455/0.462,0.462/0.5,0.5/0.667,0.667/1.0,0.5/0.6,0.6/0.667,0.5/0.571,0.571/0.6,0.5/0.556,0.556/0.571,0.5/0.545,0.545/0.556,0.5/0.538,0.538/0.545,0.545/0.55,0.55/0.556,0.556/0.562,0.562/0.571,0.571/0.583,0.583/0.6,0.6/0.625,0.625/0.667,0.6/0.615,0.615/0.625,0.625/0.636,0.636/0.667,0.667/0.75,0.75/1.0,0.667/0.714,0.714/0.75,0.667/0.7,0.7/0.714,0.714/0.727,0.727/0.75,0.75/0.8,0.8/1.0,0.75/0.778,0.778/0.8,0.75/0.769,0.769/0.778,0.778/0.786,0.786/0.8,0.8/0.833,0.833/1.0,0.8/0.818,0.818/0.833,0.833/0.857,0.857/1.0,0.833/0.846,0.846/0.857,0.857/0.875,0.875/1.0,0.857/0.867,0.867/0.875,0.875/0.889,0.889/1.0,0.875/0.882,0.882/0.889,0.889/0.9,0.9/1.0,0.889/0.895,0.895/0.9,0.9/0.909,0.909/1.0,0.9/0.905,0.905/0.909,0.909/0.917,0.917/1.0,0.909/0.913,0.913/0.917,0.917/0.923,0.923/1.0,0.917/0.92,0.92/0.923,0.923/0.929,0.929/1.0,0.923/0.926,0.926/0.929,0.929/0.933,0.933/1.0,0.929/0.931,0.931/0.933,0.933/0.938,0.938/1.0,0.933/0.935,0.935/0.938,0.938/0.941,0.941/1.0,0.938/0.939,0.939/0.941,0.941/0.944,0.944/1.0,0.941/0.943,0.943/0.944,0.944/0.947,0.947/1.0,0.944/0.946,0.946/0.947,0.947/0.95,0.95/1.0,0.947/0.949,0.949/0.95,0.95/0.952,0.952/1.0,0.95/0.951,0.951/0.952,0.952/0.955,0.955/1.0}{
    	\foreach \n in {-3,-2,-1,0,1,2} {
	    	\draw[] ({\n+\rightend},0) arc (0:180:{0.5*(\rightend-\leftend)}); 
	    	\draw[] ({-\n-\leftend},0) arc (0:180:{0.5*(\rightend-\leftend)});  
	 } 	
   }
   
    \foreach \bottom in {-3,-2,-1,0,1,2,3} {
    	\draw (\bottom,0) -- (\bottom,1.5);
    }
   
\foreach \pos/\label in {-2/$-2\phantom{-}$,-1/$-1\phantom{-}$,0/$0$,1/$1$,2/$2$,
-2.5/$-\!\tfrac52\phantom{-}$,-1.5/$-\!\tfrac32\phantom{-}$,-0.5/$-\!\tfrac12\phantom{-}$,0.5/$\tfrac12$,1.5/$\tfrac32$,2.5/$\tfrac52$,-2.667/$-\!\tfrac83\phantom{-}$,-2.333/$-\!\tfrac73\phantom{-}$,-1.667/$-\!\tfrac53\phantom{-}$,-1.333/$-\!\tfrac43\phantom{-}$,-0.667/$-\!\tfrac23\phantom{-}$,-0.333/$-\!\tfrac13\phantom{-}$,0.333/$\tfrac13$,0.667/$\tfrac23$,1.333/$\tfrac43$,1.667/$\tfrac53$,2.333/$\tfrac73$,2.667/$\tfrac83$}{
    	\node [below] at (\pos,0) {{\footnotesize \label}}; 	
   }

\begin{scope}[draw=alphacol,very thick]
\draw[midarrow=5pt] (1,1.5) -- (1,0);
\draw[] (1,0) arc (180:0:0.25);\draw[midarrow=4pt,draw=none](1.2,0.25) -- (1.4,0.25);
\draw[] (1.5,0) arc (180:0:0.083);\draw[midarrow=3.5pt,draw=none](1.57,0.083) -- (1.67,0.083);
\draw[] (1.666,0) arc (0:180:0.033);

\end{scope}
	
\end{tikzpicture}
\caption{Path in the Farey graph}
\label{figure1}
\end{figure}

Each element $(b_0,b_1,\dots,b_n)\in \widehat{O}\setminus\{0\}$ determines a continued fraction whose convergents form a path from $\infty$ to $0$ in the Farey graph, without intersections. By adjoining the edge from $0$ to $\infty$ we obtain a simple closed path, a loop, traversed clockwise if the $b_i$ are positive and anticlockwise if they are negative, which encloses an ideal triangulated polygon. We will make use of the triangulated polygons of Figure~\ref{figure2}, depicted as Euclidean rather than hyperbolic polygons.

\begin{figure}[H]
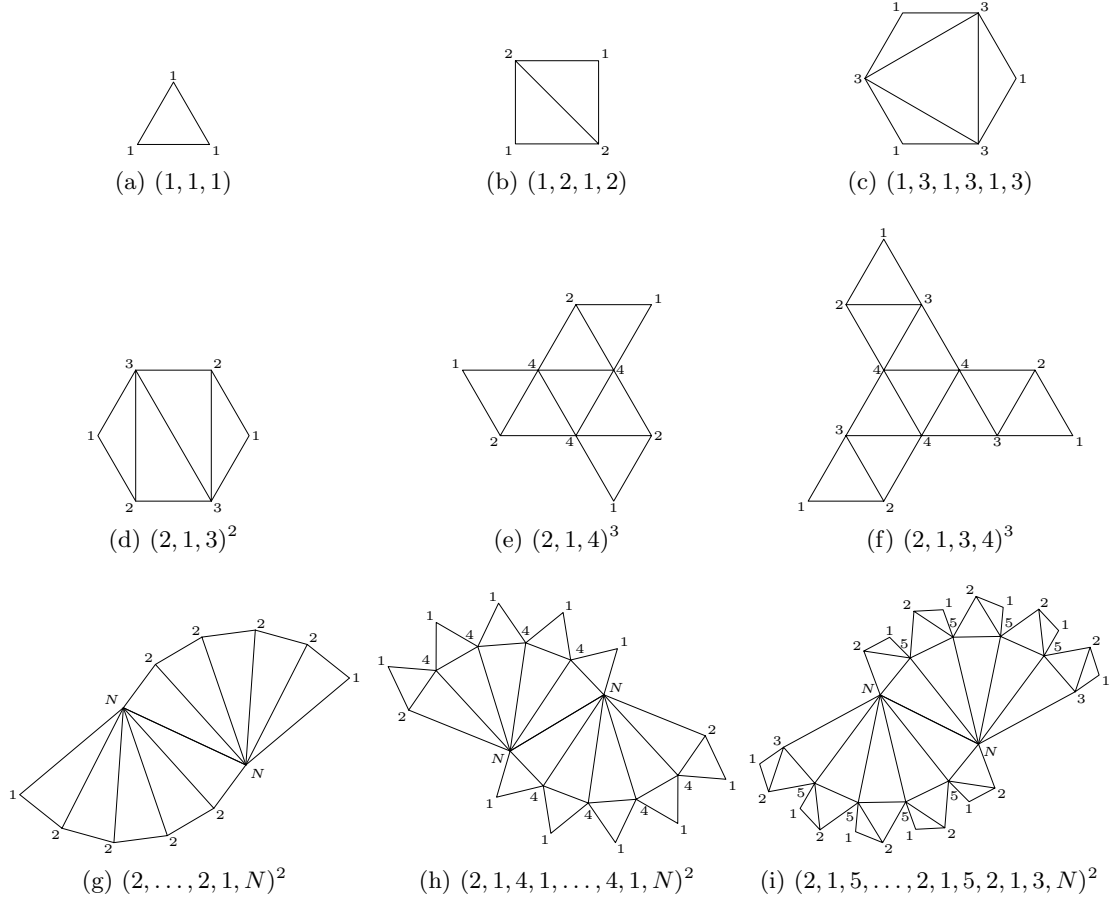

\centering
\PolySub{\PolyA}{$(1,1,1)$}\hfill
\PolySub{\PolyB}{$(1,2,1,2)$}\hfill
\PolySub{\PolyC}{$(1,3,1,3,1,3)$}

\vspace{1em}
\PolySub{\PolyD}{$(2,1,3)^2$}\hfill
\PolySub{\PolyE}{$(2,1,4)^3$}\hfill
\PolySub{\PolyF}{$(2,1,3,4)^3$}

\vspace{1em}
\PolySub[.33\linewidth]{\PolyG}{$(2,\dots,2,1,N)^2$}\hfill
\PolySub[.33\linewidth]{\PolyH}{$(2,1,4,1,\dotsc,4,1,N)^2$}\hfill
\PolySub[.33\linewidth]{\PolyI}{$(2,1,5,\dots,2,1,5,2,1,3,N)^2$}

\caption{Triangulated polygons}
\label{figure2}
\end{figure}

The vertex labels in Figure~\ref{figure2} represent the number of triangles incident to each vertex. The label at vertex $j$ is equal to the modulus of the coefficient $b_j$ of the corresponding continued fraction. The captions describe the resulting sequences of labels, read clockwise. Here $x^m$ denotes $m$-fold concatenation of the word $x$. Sequences of vertex labels of this type were called \emph{quiddity sequences} by Conway and Coxeter in \cite{CoCo1973} in their study of frieze patterns. The set $\widehat{O}^+$ of positive elements of $\widehat{O}$ consists of \emph{trimmed quiddity sequences}: each element of $\widehat{O}^+$ is obtained by removing the last coefficient of a quiddity sequence. For example, $(1,2,1,2)$ is a quiddity sequence and $(1,2,1)$ is a trimmed quiddity sequence. It is easily shown that $\widehat{O}=\widehat{O}^+\cup(-\widehat{O}^+)\cup\{0\}$.	

We prove Theorem~\ref{theoremA} by identifying substrings of trimmed quiddity sequences. This process is removed from the theory of continued fractions, giving us the following independent corollary. Here we say that a finite set $X$ of positive integer words without redundancy is an \emph{unavoidable set} if each trimmed quiddity sequence has a substring that belongs to $X$. The notion of being unavoidable is similar to that of exclusion; we use the former term only in the context of this corollary. There is a partial order $\preceq$ on finite unavoidable sets for which $X\preceq Y$ if every $x\in X$ contains some $y\in Y$ as a substring. A \emph{minimal} unavoidable set of length $L$ is a minimal element in the poset of finite unavoidable sets of length $L$.

\begin{maincorollary}{C}\label{corollaryC}
The minimal unavoidable sets of length two are precisely $O_1,O_2,\dots,O_9$ and their reversals.
\end{maincorollary}

In \cite{Cu2018}, Cuntz describes an algorithm for obtaining (not minimal) unavoidable sets of quiddity cycles (rather than trimmed quiddity sequences) of any prescribed length. Corollary~\ref{corollaryC} is of a similar spirit; it completely classifies all minimal unavoidable sets, although only of length two. Cuntz's algorithm, or that of Section~\ref{section2}, could be used to compute minimal unavoidable sets (and exclusion sets) of any given length, subject to the limits of computation.

Let us now turn to exclusion sets for Gaussian integer continued fractions. Here we consider only \emph{reversible} exclusion sets of length two. To save notation, we often write reversible sets of singletons and pairs as collections of multisets of cardinality 1 or 2. For example, we write $X=\{ \{0\}, \{1,i\}, \{1,1\} \}$ rather than $X=\{ 0, (1,i),(i,1), (1,1)\}$. A \emph{minimal} reversible exclusion set of length two is a minimal element in the poset of all finite reversible exclusion sets of length two. Remarkably, we can classify all such minimal sets.

\begin{maintheorem}{D}\label{theoremD}
There are exactly two minimal  reversible  exclusion sets over the Gaussian integers of length two, namely $\overline{U_i}$ where
\[
\begin{aligned}
U_1
&=
\left\{
\begin{aligned}
&\{1,1\}, \{1,2\}, \{1,3\}, \{1,i\}, \{1,1+i\},
 \{1,2+i\}, \{1,3+i\},\\
&\{1,1+2i\}, \{1,2+2i\}, \{1+i,1-i\},
 \{1+i,2\}, \{1+i,2-i\}
\end{aligned}
\right\}\\[1ex]
U_2
&=
\left\{
\begin{aligned}
&\{1,1\}, \{1,2\}, \{1,3\}, \{1,i\}, \{1,1+i\},
 \{1,2+i\}, \{1,3+i\},\\
&\{1+i,2+i\},\{1,2+2i\}, \{1+i,1-i\},
 \{1+i,2\}, \{1+i,2-i\}
\end{aligned}
\right\}.
\end{aligned}
\]
\end{maintheorem}

The full set $\overline{U_1}$ is far larger than $U_1$; for instance, in place of the pair $\{1+i,2-i\}$ in $U_1$ we have the collection of pairs
\[
\pm\{1+i,2-i\}, \quad \pm\{1-i,2+i\}, \quad \pm\{1+i,1-2i\}, \quad  \pm\{1-i,1+2i\} 
\]
in $\overline{U_1}$ . Notice that all these pairs have product $3\pm i$, because the operation $\{a,b\}\mapsto \{ua, u^{-1}b\}$ leaves the product unchanged.

The sets $U_1$ and $U_2$ differ only in the pairs $\{1,1+2i\}$ and $\{1+i,2+i\}$. The significance of this is explained, in part, by the periodic continued fraction
\[
[\overline{S,-S}],\quad\text{where } S=(1+2i,1,1-2i,-1+i,-2+i,-i,2+i,1+i)
\]
(and $[\overline{x}]$ denotes the periodic continued fraction with period $x$). Every adjacent pair $\{b_{i-1},b_i\}$ from this continued fraction is equal to one of $\{1,1+2i\}$ or $\{1+i,2+i\}$, up to complex conjugation and the $\{ua,u^{-1}b\}$ operation. We illustrate this and other (later) Gaussian integer continued fractions by paths in the Gaussian Farey graph, in the same way that we display integer continued fractions in the usual Farey graph. The Gaussian Farey graph is the infinite graph in upper half three-space with vertices at Gaussian rationals (and infinity) and with an edge between two reduced Gaussian rationals $a/b$ and $c/d$ if $|ad-bc|=1$ (see, for example, \cite{Ho2020} for background). The periodic continued fraction $[\overline{S,-S}]$ is illustrated by the closed path that passes through the 16 convergents of the continued fraction in Figure~\ref{figure74} (after applying a suitable M\"obius transformation to give a more symmetric picture). These are Gaussian rationals, marked by solid spots, with labels outside the path. The smaller labels inside are the continued fraction coefficients, all Gaussian integers.

\begin{figure}[ht]
\centering
\begin{tikzpicture}[scale=1.5]
  \GridThreeThree
  \Edge{(1,1)}{(1,0)} \Edge{(1,0)}{(1.5,-0.5)} \Edge{(1.5,-0.5)}{(1,-1)} \Edge{(1,-1)}{(1,-2)} \Edge{(1,-2)}{(0,-2)} \Edge{(0,-2)}{(-0.5,-2.5)} \Edge{(-0.5,-2.5)}{(-1,-2)} \Edge{(-1,-2)}{(-2,-2)} 
  \Edge{(-2,-2)}{(-2,-1)} \Edge{(-2,-1)}{(-2.5,-0.5)} \Edge{(-2.5,-0.5)}{(-2,0)} \Edge{(-2,0)}{(-2,1)} \Edge{(-2,1)}{(-1,1)} \Edge{(-1,1)}{(-0.5,1.5)}\Edge{(-0.5,1.5)}{(0,1)} \Edge{(0,1)}{(1,1)}
	  \foreach \p in {(1,1),(1,0),(1.5,-0.5),(1,-1),(1,-2),(0,-2),(-0.5,-2.5),(-1,-2),(-2,-2),(-2,-1),(-2.5,-0.5),(-2,0),(-2,1),(-1,1),(-0.5,1.5),(0,1)} {\Spot{\p}}
  \node[rational,above right=2pt] at (1,1) {$1+i$};
  \node[rational,right=2pt] at (1,0) {$1$}; 
  \node[rational,right=2pt] at (1.5,-0.5) {$\tfrac32-\tfrac12i$};
  \node[rational,right=2pt,yshift=-1pt] at (1,-1) {$1-i$};   
  \node[rational,below right=2pt] at (1,-2) {$1-2i$};
  \node[rational,below=2pt,xshift=2pt] at (0,-2) {$-2i$}; 
  \node[rational,below=2pt] at (-0.5,-2.5) {$-\tfrac12-\tfrac52i$};
  \node[rational,below left=2pt] at (-1,-2) {$-1-2i$};  
  \node[rational,below left=2pt] at (-2,-2) {$-2-2i$};
  \node[rational,left=2pt,yshift=-1pt] at (-2,-1) {$-2-i$}; 
  \node[rational,left=2pt] at (-2.5,-0.5) {$-\tfrac52-\tfrac12i$};
  \node[rational,left=2pt] at (-2,0) {$-2$};   
  \node[rational,above left=2pt] at (-2,1) {$-2+i$};
  \node[rational,above left=2pt] at (-1,1) {$-1+i$}; 
  \node[rational,above=2pt] at (-0.5,1.5) {$-\tfrac12+\tfrac32i$};
  \node[rational,above=2pt] at (0,1) {$i$};  

  \node[gint,below left=2pt] at (1,1) {$1+i$};
  \node[gint,above left=2pt] at (1,0) {$-1-2i$}; 
  \node[gint,left=2pt] at (1.5,-0.5) {$-1$};
  \node[gint,above left=2pt] at (1,-1) {$-1+2i$};   
  \node[gint,above left=2pt] at (1,-2) {$1-i$};
  \node[gint,above right=2pt] at (0,-2) {$2-i$}; 
  \node[gint,above=2pt] at (-0.5,-2.5) {$i$};
  \node[gint,above right=2pt] at (-1,-2) {$-2-i$};  
  \node[gint,above right=2pt] at (-2,-2) {$-1-i$};
  \node[gint,above right=2pt] at (-2,-1) {$1+2i$}; 
  \node[gint,right=2pt] at (-2.5,-0.5) {$1$};
  \node[gint,above right=2pt] at (-2,0) {$1-2i$};   
  \node[gint,below right=2pt] at (-2,1) {$-1+i$};
  \node[gint,below right=2pt] at (-1,1) {$-2+i$}; 
  \node[gint,below=5pt] at (-0.5,1.5) {$-i$};
  \node[gint,below right=2pt] at (0,1) {$2+i$};  
 \end{tikzpicture}
 \caption{Path of $[\overline{S,-S}]$ in the Gaussian Farey graph}
 \label{figure74}
 \end{figure}
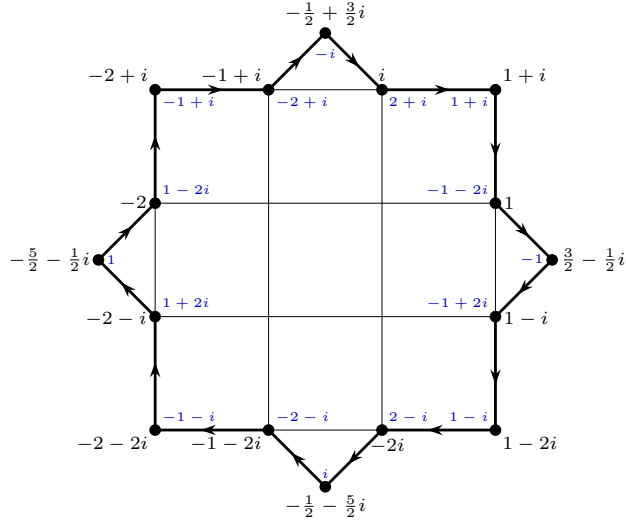

The collection $\overline{U_1}$ can be described succinctly by taking products of pairs. Specifically, let
\[
\Omega=\{1,2,3,\pm i, 1\pm i, 2\pm i, 3\pm i, 1\pm 2i, 2\pm 2i\};
\]
then
\begin{equation}\label{equation98}
\overline{U_1} = \{(a,b)\in(\mathbb{Z}[i]\setminus\{0\})^2:ab\in\Omega\}\cup\{0\}.
\end{equation}
The Gaussian integers $2$, $3\pm i$, and $2\pm 2i$ are composite, and 
\[
\hspace{-1.6cm}2=1\cdot 2=(1+i)\cdot (1-i),\quad 3+i=1\cdot (3+i)=(1+i)(2-i),\quad 2+2i=1 \cdot (2+2i)=(1+i)\cdot 2;
\] 
all other elements of  $\Omega$ are primes or units. The collection $\overline{U_2}$ cannot quite be determined by its products in the same way, because $\{1+i,2+i\}\in \overline{U_2}$, $\{1,1+3i\}\not\in \overline{U_2}$, and $1+3i=1\cdot (1+3i)=(1+i)\cdot (2+i)$.

The set $\Omega$ comprises those Gaussian integers that lie inside the cardioid $|z-1|^2<1+2|z|$, with the exception of $3\pm 2i$, which are excluded. This set is illustrated in Figure~\ref{figure5}. The solid spots mark elements of $\Omega$ and the products $1\pm 3i$ that arise from $\overline{U_2}$. 

\begin{figure}[H]
\centering
\begin{tikzpicture}[scale=0.9,>=Latex,every node/.style={font=\small}]

\def\xmin{-2.3}
\def\xmax{4.8}
\def\ymin{-3.5}
\def\ymax{3.6}

\draw[step=1, gray!35, dotted] (\xmin,\ymin) grid (\xmax,\ymax);

\fill[cyan!18, opacity=0.85]
  plot[domain=-180:180, samples=350, smooth, variable=\t]
    ({(2+2*cos(\t))*cos(\t)}, {(2+2*cos(\t))*sin(\t)})
  -- cycle;

\foreach \x in {-2,-1,0,1,2,3,4}{
  \foreach \y in {-3,-2,-1,0,1,2,3}{
    \draw[gray!55, fill=white] (\x,\y) circle (0.045);
  }
}

\draw[->, thick] (\xmin,0) -- (\xmax,0) ;
\draw[->, thick] (0,\ymin) -- (0,\ymax);

\foreach \x in {-2,-1,0,1,2,3}
  \node[below] at (\x,0) {$\x$};
\node[below left] at (4,0) {$4$};

\node[left] at (0,1) {$i$};
\node[left] at (0,-1) {$-i$};
\foreach \y in {-3,-2,2,3}
  \node[left] at (0,\y) {$\y i$};

\draw[black,thick]
  plot[domain=-180:180, samples=350, smooth, variable=\t]
    ({(2+2*cos(\t))*cos(\t)}, {(2+2*cos(\t))*sin(\t)});
    
\foreach \z in {(0,-1),(0,1),(1,-2),(1,-1),(1,0),(1,1),(1,2),(2,-2),(2,-1),(2,0),(2,1),(2,2),(3,-1),(3,0),(3,1),(1,3),(1,-3)}{
  \fill[black] \z circle (0.065);
}

\end{tikzpicture}
\caption{Cardioid $|z-1|^2<1+2|z|$ and products of pairs from $\overline{U_1}$ and $\overline{U_2}$}
\label{figure5}
\end{figure}
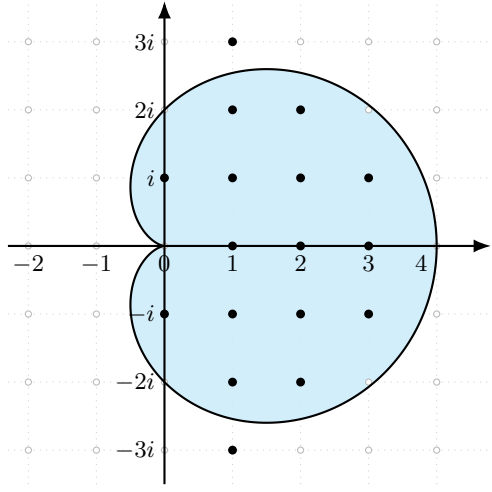

We can see from \eqref{equation98} that $\overline{U_1}$ can be obtained from the exclusion set $U$ of \eqref{equation99}  by removing all pairs with product $3\pm 2i$. Since the continued fraction $[\overline{1,3+2i}]$ is excluded by $U$, but not by $\overline{U_1}$, it follows that $\overline{U_1}$ strengthens the classical criterion obtained from the parabola theorem.

\begin{maincorollary}{E}
We have $\Ex(\overline{U_1})\subsetneq \Ex(U)$.
\end{maincorollary}

This corollary fails with $\overline{U_2}$ in place of $\overline{U_1}$ because, for example, $[\overline{1+i,2+i}]$ is excluded by $\overline{U_2}$ but not by $U$.

\section{Exclusion sets of length two over the integers}\label{section2}

In this section we prove Theorem~\ref{theoremA}, the first part of Theorem~\ref{theoremB}, and Corollary~\ref{corollaryC}. We begin this task by reviewing  some standard concepts from continued fraction theory, introducing notation that we will use throughout. Let $\widehat{\mathbb{C}}$ denote the Riemann sphere $\mathbb{C}\cup\{\infty\}$. Let $b_0,b_1,\dotsc$ be integers or Gaussian integers and let $A_0,A_1,\dotsc$ and $B_0,B_1,\dotsc$ be given by
\begin{equation*}\label{B}
\begin{pmatrix}A_{n+1} & -A_{n}\\B_{n+1} & -B_{n}\end{pmatrix}
=M_{b_0}M_{b_1}\dotsb M_{b_n},\quad\text{where } M_{b}=\begin{pmatrix}b & -1\\1 & 0\end{pmatrix}.
\end{equation*}
Then $A_0=1$, $B_0=0$, $A_1=b_0$, $B_1=1$, and $A_{n+1}=b_nA_n-A_{n-1}$ and $B_{n+1}=b_nB_n-B_{n-1}$, for $n\geq 1$. The quotients $A_n/B_n\in\widehat{\mathbb{C}}$ are the convergents of the continued fraction $[b_0,b_1,\dotsc]$. We consider convergence with respect to the topological space $\widehat{\mathbb{C}}$; thus convergence to $\infty$ is permitted.

\begin{lemma}\label{lemma83}
The set $\widehat{O}$ is an exclusion set.
\end{lemma}
\begin{proof}
The set $\widehat{O}$ is symmetric because $[-b_0,-b_1,\dots,-b_k]=-[b_0,b_1,\dots,b_k]$, and it is without redundancy because if one member of $\widehat{O}$ were a proper substring of another, then the corresponding subpath would begin and end at the same vertex, which contradicts the absence of repeated vertices.

Now let $[b_0,b_1,\dotsc]$ be an integer continued fraction that diverges. We write the convergents as $V_i=A_i/B_i$. It is known, and straightforward to prove, that infinitely many $V_i$ are equal (see \cite{ShSt2022}*{Theorem~1.3}). Thus, for any positive integer $n$, we can find integers $r$ and $s$ with $n<r<s$ for which $V_r,V_{r+1},\dots,V_{s-1}$ are pairwise distinct and $V_r=V_s$. By applying the matrix 
\[
\begin{pmatrix}
-B_{r} & A_{r}\\
-B_{r+1} & A_{r+1}
\end{pmatrix}
\]
to $V_{r+1},V_{r+2},\dots,V_s$ we see that the continued fraction $[b_{r+1},b_{r+2},\dots,b_{s-1}]$ has value 0 and no repeated convergents. Hence $(b_{r+1},b_{r+2},\dots,b_{s-1})\in\widehat{O}$, so every tail of $[b_0,b_1,\dotsc]$ has a substring that belongs to $\widehat{O}$, as required.
\end{proof}

A \emph{rooted triangulated polygon} is a triangulated polygon with a distinguished oriented boundary edge. Each rooted triangulated polygon is specified uniquely by a triple $A=(a;u;c)$, where $a$ and $c$ are the vertex degrees at the initial and final vertices of the oriented edge, and $u=(u_1,u_2,\dots,u_m)$ is the list of vertex degrees between $a$ and $c$ obtained by traversing the boundary of the polygon in the opposite direction to the orientation of the edge. We also allow a degenerate triple $E=(0;\varnothing;0)$. Let  $\tau(A)=(a,u)$; this is a trimmed quiddity sequence (for $A\neq E$), and every trimmed quiddity sequence arises in this way.

Given two rooted triangulated polygons $A=(a;u;c)$ and $B=(b;v;d)$ we define
\[
A\ast B = (a+1;u,b+c+1,v;d+1),
\]
which is the rooted triangulated polygon obtained by gluing $A$ and $B$ along their distinguished edges to two edges of an oriented triangle, with opposite orientations matching. There is a well-known correspondence between rooted triangulated polygons and rooted planar binary trees, and this operation is the usual grafting operation for binary trees considered in, for example, \cite{AgSo2006}.  Each nondegenerate rooted triangulated polygon $C$ has a unique decomposition $C=A\ast B$ (where $A$ or $B$ may be $E$), and consequently every rooted triangulated polygon can be obtained from products of $E$ (note that $\ast$ is not associative). We define
\[
F=E\ast E=(1;1;1),\quad G=E\ast F=(1;2,1;2),\quad H=G\ast E=(2;2,1,3;1).
\]
Let $\eta(C)$ denote the number of triangles in a rooted triangulated polygon. Then $\eta(A\ast B) = \eta(A)+\eta(B)+1$, $\eta(E)=0$, $\eta(F)=1$, $\eta(G)=2$, and $\eta(H)=3$.

Given a set $X$ of ordered pairs of integers, we say that $A=(a;u;c)$ is \emph{$X$-free} if $u$ has no substring that belongs to $X$. Notice that if $A\ast B$ is $X$-free, then both $A$ and $B$ are $X$-free. We also describe a trimmed quiddity sequence as $X$-free if it has no substring that belongs to $X$.

\begin{lemma}\label{lemma70}
Every trimmed quiddity sequence of a triangulated polygon has a substring that belongs to $O_i$, for each $i=1,2,\dots,9$.
\end{lemma}
\begin{proof}
We start with $O_1$. Suppose, in order to reach a contradiction, that $C$ is an $O_1$-free rooted triangulated polygon of at least two triangles for which $\eta(C)$ is minimal. We can write $C=A\ast B$, where $A$ and $B$ are $O_1$-free and $\eta(A), \eta(B)<\eta(C)$, so we must have $\eta(A),\eta(B)\leq 1$. Hence $A,B\in\{E,F\}$. However, none of $E\ast F$, $F\ast E$, and $F\ast F$ are $O_1$-free (and $\eta(E\ast E)=1$), so we have a contradiction. Consequently, the only nondegenerate rooted triangulated polygon that is $O_1$-free is $F$. But $\tau(F)=(1,1)\in O_1$, so no trimmed quiddity sequence is $O_1$-free.

Next consider $O_i$, for $i\in\{2,3\}$. A quick check shows that the only  $O_i$-free rooted triangulated polygons $C$ with $\eta(C)\leq 2$ are $E$, $F$, and $G$.  Suppose, in order to reach a contradiction, that $C$ is  an $O_i$-free rooted triangulated polygon of at least three triangles for which $\eta(C)$ is minimal. We can write $C=A\ast B$, where $A$ and $B$ are $O_i$-free and $\eta(A),\eta(B)\leq 2$. Then $A,B\in\{E,F,G\}$. However, no such products with $\eta(C)\geq 3$ are $O_i$-free, which is a contradiction. Consequently, the only nondegenerate rooted triangulated polygons that are $O_i$-free are $F$ and $G$. But $\tau(F)=(1,1)$ and $\tau(G)=(1,2,1)$, and $(1,1),(1,2)\in O_i$, so no trimmed quiddity sequence is $O_i$-free.

Last we consider $O_i$, for $i\in\{4,5,\dots,9\}$. This time we observe that the only $O_i$-free rooted triangulated polygons $C$ with $\eta(C)\leq 3$ are $E$, $F$, $G$, and $H$. Arguing as before, we see that in fact these are the only rooted triangulated polygons that are $O_i$-free.  Since $\tau(F)=(1,1)$, $\tau(G)=(1,2,1)$, $\tau(H)=(2,2,1,3)$, and  $(1,1),(1,2),(2,2)\in O_i$, we deduce that no trimmed quiddity sequence is $O_i$-free.
\end{proof}

An alternative approach to Lemma~\ref{lemma70} would be to use Cuntz's lists of words contained in quiddity cycles described in \cite{Cu2018}. 

We can now prove Theorem~\ref{theoremA}.

\begin{proof}[Proof of Theorem~A]
First we prove that $\overline{O_i}$ is an exclusion set, for each $i\in\{1,2,\dots,9\}$. Consider a divergent continued fraction $[b_0,b_1,\dotsc]$. By Lemma~\ref{lemma83}, each tail contains a substring $s$ that belongs to $\widehat{O}$. If $s=0$, then $s\in \overline{O_i}$. If $s\neq 0$, then one of $\pm s$ is a trimmed quiddity sequence, so $s$ itself has a substring that belongs to $\overline{O_i}$, by Lemma~\ref{lemma70}. Hence $[b_0,b_1,\dotsc]\in \Ex(\overline{O_i})$.

Since the collection of triangulated polygons is closed under reflection, we see that the collection of trimmed quiddity sequences is closed under reversal. Hence we can apply the argument of the preceding paragraph with the reversal of $s$ to deduce that the reversals of the $\overline{O_i}$ are also exclusion sets.

It remains to prove that $\overline{O_1},\overline{O_2},\dots,\overline{O_9}$ are the \emph{only} minimal exclusion sets of length two. Suppose then that $O$ is some finite exclusion set of length two; we will prove that $\overline{O}_i\preceq O$, for some $i$, up to reversal. For this we use the triangulated polygons (a)--(i) of Figure~\ref{figure2}. The captions give the quiddity sequences, which are periods of periodic (divergent) continued fractions. For example, from (f) we have that $[\overline{2,1,3,4}]$ is a divergent continued fraction.

Since $O$ is an exclusion set, we see from (a) that $O$ contains $1$ or $(1,1)$. If $1\in O$, then $\overline{O}_1\preceq O$, so let us suppose that $1\notin O$, in which case $(1,1)\in O$. From (b) and (c) we see that $O$ contains one of $2$, $(1,2)$, or $(2,1)$ and one of  3, $(1,3)$, or $(3,1)$. Suppose that $2\in O$. Then we must have $\overline{O}_1\preceq O$, up to reversal. The alternative, which we assume henceforth, is that $2\notin O$. Let us also assume that $(1,2)\in O$ (the alternative $(2,1)\in O$ leads to a similar argument with words reversed).

Suppose that $(2,1)\in O$. We also assume that either $3\in O$ or $(1,3)\in O$ (again, $(3,1)\in O$ leads to a similar argument with reversal). In both cases we have $\overline{O}_1\preceq O$. Let us assume then that $(2,1)\notin O$. From (g), since $N$ can be chosen arbitrarily, we see that $(2,2)\in O$.  Also, by (h) we have $4\in O$, $(1,4)\in O$, or $(4,1)\in O$, and by (e)  we have $4\in O$, $(1,4)\in O$, or $(4,2)\in O$. If $3\in O$ or $(1,3)\in O$, then either $\overline{O_2}\preceq O$ or $\overline{O_3}\preceq O$.

Suppose instead that $3,(1,3)\notin O$. Then $(3,1)\in O$ (by (c)) and $(3,2)\in O$ (by (d)). Then, by (f) we have $4\in O$, $(3,4)\in O$, or $(4,2)\in O$, and by (i) we have $5\in O$, $(1,5)\in O$, or $(5,2)\in O$. If $4\in O$, then either $\overline{O_4}\preceq O$ or $\overline{O_5}\preceq O$, so let us assume $4\notin O$. The remaining possibility is that the following four restrictions hold: $(1,4)\in O$ or $(4,1)\in O$; $(1,4)\in O$ or $(4,2)\in O$; $(3,4)\in O$ or $(4,2)\in O$; and $5\in O$, $(1,5)\in O$, or $(5,2)\in O$. These give us the six  cases $\overline{O_i}\preceq O$, for $i=4,5,\dots,9$.

Finally, we prove that $\overline{O_1},\overline{O_2},\dots,\overline{O_9}$ and their reversals are pairwise incomparable. Consider any two of these sets, $X$ and $Y$. Neither set is contained in the other, so we may choose $s\in X\setminus Y$. Then the continued fraction $[1,s,2,s,3,s,\dotsc]$ belongs to $\Ex(X)$ but not $\Ex(Y)$. Repeating this argument in the opposite direction proves that $X$ and $Y$ are incomparable. Hence these eighteen sets are the full collection of minimal exclusion sets.
\end{proof}

Corollary~\ref{corollaryC} is closely related to Theorem~\ref{theoremA}, and the proofs of the two results are similar.

\begin{proof}[Proof of Corollary~\ref{corollaryC}]
Lemma~\ref{lemma70} tells us that every trimmed quiddity sequence contains a member of $O_i$, and since the collection of trimmed quiddity sequences is reversible we see that each such sequence also contains a member of the reversal of $O_i$. Hence the sets $O_i$ and their reversals are unavoidable sets of length two.

To see that the sets $O_i$ and their reversals are the full collection of minimal unavoidable sets of length two, we can use the nine quiddity sequences of Figure~\ref{figure2}, following the argument for Theorem~\ref{theoremA}. In this case it is immediate that the sets $O_i$ and their reversals are pairwise incomparable, so they form the full collection of minimal elements.
\end{proof}

We finish this section with a brief discussion of exclusion sets of length three. Our experiments suggest that there are hundreds of minimal exclusion sets over the integers of length three, each containing at least thirteen positive tuples, and that $\overline{O}$ is one of these exclusion sets, where
\[
\begin{aligned}
O=\{& (1,1),(1,2,1),(2,1,2),(1,2,2),(1,3,1),(2,1,3),(2,1,4),\\
&(2,2,1),(3,1,2),(4,1,2),(1,2,3),(3,2,1),(1,5,1)\}. 
\end{aligned}
\]
These experiments suggest the following conjecture.

\begin{conjecture}
Among minimal integer exclusion sets of length three, the minimum number of positive tuples is thirteen. There are exactly nine sets attaining this minimum, and $\overline{O}$ is the only reversible one.
\end{conjecture}

\section{Infinite exclusion sets over the integers}

In this section we complete the proof of Theorem~\ref{theoremB}.

\begin{lemma}\label{lemma7}
Let $s=(b_0,b_1,\dots,b_m)\in\widehat{O}$, where $s\neq 0$. Then $q=1/[b_0,b_1,\dots,b_{m-1}]$ is an integer, and the continued fraction
\[
[q+1,s,q-1,s,q+2,s,q-2,s,q+3,s,q-3,\dotsc]
\]
diverges.
\end{lemma}
\begin{proof}
We have $[b_0,b_1,\dots,b_m]=0$, so $A_{m+1}=0$ and $B_{m+1}=\varepsilon$, where $\varepsilon\in\{-1,1\}$, in which case
\[
M_{b_0}M_{b_1}\dotsb M_{b_m}
=
\begin{pmatrix}A_{m+1} & -A_{m}\\B_{m+1} & -B_{m}\end{pmatrix}
=Q,\quad\text{where } Q=\varepsilon \begin{pmatrix}0 & -1\\1 & -x\end{pmatrix},
\]
for some integer $x$. Then $A_m=\varepsilon$ and $B_m=\varepsilon x$, so $x=q$. Now, for  $k,k'\in\mathbb{Z}$, we have  
\(
M_kQM_{k'}=-\varepsilon M_{k+k'-q}.
\)
It follows that 
\[
M_{q+1}QM_{q-1}QM_{q+2}QM_{q-2}Q\dotsb M_{q+n}QM_{q-n} = -\varepsilon\begin{pmatrix}q & -1\\1 & 0\end{pmatrix}.
\]
Hence there are infinite subsequences of convergents equal to $q$ and $\infty$, so the continued fraction diverges.
\end{proof}

We can now prove Theorem~\ref{theoremB}.

\begin{proof}[Proof of Theorem~\ref{theoremB}]
Lemma~\ref{lemma83} tells us that $\widehat{O}$ is an exclusion set. It remains to prove that $\Ex(\widehat{O})\subsetneq \Ex(O)$ for any finite exclusion set $O$. For this, observe first that $\widehat{O}$ and $O$ both contain $0$. Now choose any element $s\in\widehat{O}$ with $s\neq 0$ and consider the divergent continued fraction of Lemma~\ref{lemma7}. Choose a sufficiently large positive integer $n$  that every integer coefficient that appears in the (finite) exclusion set $O$ has modulus less than $n-|q|$. Now, we know that $[q+n,s,q-n,s,q+n+1,s,q-n-1,\dotsc]$ contains a substring that belongs to $O$. Since all the coefficients of this continued fraction besides those of $s$ have modulus at least $n-|q|$, we deduce that $O$ contains an element that is a substring of $s$. The element $s$ was chosen arbitrarily from $\widehat{O}$, so $\Ex(\widehat{O})\subset \Ex(O)$.

Consider now an element $t$ of $\widehat{O}$ of length greater than any element of $O$. Let $s\in O$ be a substring of $t$, and consider the continued fraction $[N_1,s,N_2,s,N_3,s,\dotsc]$, where $N_1,N_2,\dotsc$ is an increasing sequence of positive integers, with $N_j$ chosen recursively as follows. Suppose that $N_1,N_2,\dots,N_{j-1}$ have been defined. Let us write $s=(s_0,s_1,\dots,s_m)$. We define $F_b(z)=b-1/z$ and $G=F_{s_0}F_{s_1}\dotsb F_{s_m}$ (with functional composition). Let $H=F_{N_1}GF_{N_2}G\dotsb F_{N_{j-1}}G$. Now, the rationals $N_j$ and $N_j-1/[s_0,s_1,\dots,s_k]$, for $k=0,1,\dots,m$, are distinct, because $s$ is a substring of an element of $\widehat{O}$, and by choosing $N_j$ to be sufficiently large we can assume that these rationals are distinct from $H^{-1}(V_i)$ for any of the convergents $V_i$ of $[N_1,s,N_2,s,\dotsc,N_{j-1},s]$ determined already. It follows that $H(N_j)$ and $H(N_j-1/[s_0,s_1,\dots,s_k])$, which are the new convergents, are pairwise distinct and are distinct from the preceding convergents. Therefore $[N_1,s,N_2,s,N_3,s,\dotsc]$ has no repeated convergents, so it does not belong to $\Ex(\widehat{O})$. However, it does belong to $\Ex(O)$, so $\Ex(\widehat{O})\subsetneq \Ex(O)$, as required. 
\end{proof}

We remark that it is relatively straightforward to construct infinite exclusion sets over the integers that exclude strictly fewer convergent continued fractions than $\widehat{O}$. 

\section{Gaussian integer exclusion sets}\label{section39}

The proof of Theorem~\ref{theoremD} requires a delicate analytic argument, which will occupy much of our attention for this section. We define $G_e(z)=1-1/(ez)$, where $e\in\mathbb{Z}[i]\setminus\{0\}$, and we let
\[
\Lambda = \{0,1,2,3,\pm i, 1\pm i, 2\pm i, 3\pm i, 2\pm 2i\}.
\]
This set has the property that if $a,b\neq 0$ and $ab\in \Lambda$, then $\{a,b\}\in \overline{U_1}\cap\overline{U_2}$. We define $H=\{z:\Re z>1/2\}$, a half-plane, and we define $D(c,r)=\{z:|z-c|<r\}$, the open disc centred at $c$ of radius $r$. The following four lemmas will facilitate a nesting argument, familiar to continued fraction theory, which will allow us to prove that $\overline{U_1}$ and $\overline{U_2}$ are exclusion sets. 

\begin{lemma}\label{lemma50}
Let $e\in\mathbb{Z}[i]\setminus\Lambda$. If $e\notin \{1\pm 2i, 3\pm 2i\}$, then $G_e(H)\subset H$.
\end{lemma}
\begin{proof}
We have $G_e(H)=D(1-1/e,1/|e|)$. This lies in $H$ if and only if $1-\Re(1/e)-1/|e|\geq 1/2$, or, equivalently, if and only if $|e-1|^2\geq 1+2|e|$. This inequality describes the complement of the cardioid of Figure~\ref{figure5}. Notice that 
\[
|e-1|^2-1-2|e|=|e|^2-2\Re e-2|e|\geq |e|(|e|-4).
\]
Thus the inequality $|e-1|^2\geq 1+2|e|$ holds whenever $|e|\geq 4$, and if $|e|<4$ and $e\notin \Lambda$, then the inequality holds for all values of $e$ besides $1\pm 2i$ and $3\pm 2i$.
\end{proof}

In the next lemma we write $G_pG_q$ for the functional composition $G_p\circ G_q$, and we use similar notation later on.

\begin{lemma}\label{lemma51}
Let $p,q\in\mathbb{Z}[i]\setminus\Lambda$. Then $G_pG_q(H)\subset H$ if and only if $|pq-p-q|\geq |p|+|q|$.
\end{lemma}
\begin{proof}
We have $G_q(H)=D(1-1/q,1/|q|)$. Also, $G_p(w)\in H$ if and only if $|G_p(w)-1|<|G_p(w)-0|$, and this inequality can be simplified to give $|pw-1|>1$. Hence $G_pG_q(H) \subset H$ if and only if $|p(1-1/q+\zeta/|q|)-1|>1$ for all $|\zeta|<1$. This holds if and only if $|pq-p-q|\geq |p|+|q|$.
\end{proof}

\begin{lemma}\label{lemma52}
Let $p\in \{1\pm 2i, 3\pm 2i\}$ and $d\in\mathbb{Z}[i]\setminus\{0\}$, and let $q=pd$. If $\{p,d\}\not\in \overline{U_2}$, then 
\[
G_pG_q(H)\subset H\quad\text{and}\quad G_qG_p(H)\subset H.
\] 
\end{lemma}
\begin{proof}
Observe that $p\notin\Lambda$. Additionally, $q\notin\Lambda$ because $\{p,d\}\not\in \overline{U_2}$. Therefore, by Lemma~\ref{lemma51}, it suffices to prove that $|d(p-1)-1|\geq |d|+1$. This can quickly be verified when $p=3\pm 2i$, splitting the cases $|d|=1$ and $|d|\geq \sqrt{2}$. Suppose that $p=1\pm 2i$. Again, the inequality is easily verified if $|d|\geq 2$. If $|d|$ is 1 or $\sqrt{2}$, then, by checking the few cases, we see that the inequality only fails for $\{p,d\}\in \overline{U_2}$, as required.
\end{proof}

\begin{lemma}\label{lemma53}
Let $p,r\in \{1\pm 2i, 3\pm 2i\}$, let $u$ be a unit, and let $q=upr$. Then $G_pG_qG_r(H)\subset H$.
\end{lemma}
\begin{proof}
A direct calculation gives 
\[
G_qG_r(H) = D\mleft(1-\frac{|r|^2-r}{q\Delta_r},\frac{|r|}{|q|\Delta_r}\mright),
\]
where $\Delta_r=|r-1|^2-1>0$. Also, $G_p(w)\in H$ if and only if $|pw-1|>1$. Since $q=upr$, we see that $G_pG_qG_r(H)\subset H$ provided
\[
\mleft|p-1-\frac{\bar{r}-1}{u\Delta_r}\mright|\geq 1+\frac{1}{\Delta_r},
\]
or, equivalently, $|\Delta_r(p-1)-u^{-1}(\bar{r}-1)|\geq \Delta_r+1$. Since $|p-1|\geq 2$, we have
\[
|\Delta_r(p-1)-u^{-1}(\bar{r}-1)|-\Delta_r\geq \Delta_r-|r-1|,
\]
and one can check that $\Delta_r-|r-1|\geq 1$ for $r\in \{1\pm 2i, 3\pm 2i\}$.
\end{proof}

We can now prove that $\overline{U_1}$ and $\overline{U_2}$ are exclusion sets.

\begin{proposition}\label{prop52}
The sets $\overline{U_1}$ and $\overline{U_2}$ are exclusion sets over the Gaussian integers.
\end{proposition}
\begin{proof}
We begin with $\overline{U_2}$. Let $\mathcal{S}=\{1\pm 2i, 3\pm 2i\}$. Consider a Gaussian integer continued fraction $[b_0,b_1,\dotsc]$ that does not belong to $\Ex(\overline{U_2})$; we must prove that it converges. After deleting a finite prefix and relabelling, we can assume that $b_n\neq 0$, for $n\geq 0$, and $\{b_{n-1},b_n\}\notin \overline{U_2}$, for $n\geq 1$. Let $e_n=b_{n-1}b_n$; then $e_n\notin\Lambda$. Let $F_b(z)=b-1/z$. Then $F_{b_0}F_{b_1}\dotsb F_{b_{n-1}}(b_n)=[b_0,b_1,\dots,b_{n}]$. Observe that 
\[
G_{b_{i-1}b_i}(z) = 1-\frac{1}{b_{i-1}b_iz}  = b_{i-1}^{-1}F_{b_{i-1}}(b_i z). 
\]
Then $G_{e_1}G_{e_2}\dotsb G_{e_n}(z)=b_0^{-1}F_{b_0}F_{b_1}\dotsb F_{b_{n-1}}(b_n z)$. Hence $[b_0,b_1,\dotsc]$ converges if and only if the sequence $G_{e_1}G_{e_2}\dotsb G_{e_n}(1)$ converges, and to establish the proposition we will prove that this latter sequence converges.

To enable us to apply Lemmas~\ref{lemma50} to~\ref{lemma53}, we partition $e_1,e_2,\dotsc$ into blocks as follows. Suppose $e_n\in\mathcal{S}$. Then $e_n$ is prime, so one of $b_{n-1}$ or $b_n$ is a unit. In the first case we join $e_n$ to $e_{n+1}$, and in the second we join $e_n$ to $e_{n-1}$ (and take no action  if $n=1$). All remaining elements of the sequence form singleton blocks. This process creates blocks of length at most three. Moreover, if $e_n$ is connected to both $e_{n-1}$ and $e_{n+1}$, then neither $b_{n-1}$ nor $b_n$ is a unit, so $e_n\notin\mathcal{S}$.

By deleting $b_0$ if need be (which does not affect convergence), and relabelling the resulting sequence, we can partition $e_1,e_2,\dotsc$ into blocks each of one of the four forms $(s)$, $(p,pd)$, $(pd,p)$, or $(p,upr,r)$, where $s\notin\mathcal{S}$, $p,r\in\mathcal{S}$, $d\in\mathbb{Z}[i]\setminus\{0\}$, and $u$ is a unit. Furthermore, in the middle two cases, $\{p,d\}\notin \overline{U_2}$. To see this for the first of those middle cases (the other is similar), we have $p=e_n=vb_n$, where $v=b_{n-1}$ is a unit. Then $e_{n+1}=b_nb_{n+1}=pd$, where $d=v^{-1}b_{n+1}$. Hence $\{p,d\}=\{vb_n,v^{-1}b_{n+1}\}\notin \overline{U_2}$.

Let $T$ denote one of the maps $G_s$, $G_pG_{pd}$, $G_{pd}G_p$, or $G_pG_{upr}G_r$, corresponding to these four cases. By Lemmas~\ref{lemma50}, \ref{lemma52}, and \ref{lemma53}, we have $T(H)\subset H$. We will prove that there is a compact subset $K$ of $H$ for which $T(-5)\in K$ for any choice of $d$, $p$, $r$, and $u$. (Here $-5$ is just some distinguished point in the complement of $\overline{H}$.)

Let $K=\overline{D(1,\rho)}$, where $\rho=1/(\sqrt5-1/5)<1/2$; this is a compact subset of $H$. In the first of the four cases we have $|T(-5)-1|=1/(5|s|)<\rho$, so $T(-5)\in K$. In the second case $T=G_pG_{pd}$ we have
\[
|T(-5)-1| = \mleft|\frac{1}{p+1/(5d)}\mright|\leq \frac{1}{\sqrt5 - 1/5}=\rho,
\]
so $T(-5)\in K$. The third case is similar to the second. In the fourth case $T=G_pG_{upr}G_r$. There are only finitely many such maps, so we just need to check that $T(-5)\in H$ for each one, and then we can revise our definition of $K$ by adjoining these additional points. We have
\[
T(-5)=1-\frac{u(5r+1)}{up(5r+1)-5}.
\]
Observe that $|5r+1|\geq 5|r|-1>5$. Hence
\[
|(up(5r+1)-5)-u(5r+1)|\geq |p-1||5r+1|-5\geq 2|5r+1|-5> |u(5r+1)|.
\]
By rearranging this we see that $|T(-5)|>|T(-5)-1|$, so $T(-5)\in H$, as required.

Returning to our partition of $e_1,e_2,\dotsc$, let $T_k$ denote the map $T$ associated to the $k$th block. Then $T_k$ is a M\"obius transformation that satisfies $T_k(H)\subset H$ and $T_k(-5)\in K$. We can now apply a version of the Hillam--Thron theorem from continued fraction theory (such as a conjugated version of \cite{Be2003}*{Theorem~1}, with $-5$ in place of $\infty$ and $H$ in place of a disc) to deduce that $S_n=T_1T_2\dotsb T_n$ converges uniformly on compact subsets of $H$ to a constant $c$. It remains to show that $G_{e_1}G_{e_2}\dotsb G_{e_n}(1)\to c$.

Let $C_1=\{G_e(1):e\in\mathbb{Z}[i]\setminus\Lambda\}\cup\{1\}$. We will prove that this is a compact subset of $H$. To see this, observe that $|e-1|>1$ for $e\in\mathbb{Z}[i]\setminus\Lambda$, so $|G_e(1)-1|=|1/e|<|1-1/e|=|G_e(1)|$. Thus $G_e(1)\in H$. Furthermore, $G_e(1)\to 1$ as $e\to \infty$, from which the assertion follows.

Let $C_2=\{G_pG_{upr}(1): p,r\in\mathcal{S}, u\in\{\pm 1, \pm i\}\}$, a finite set. We will prove that $C_2\subset H$. Since $G_pG_{upr}(1)=(ur(p-1)-1)/(upr-1)$, we have
\[
|G_pG_{upr}(1)|\geq \frac{|r||p-1|-1}{|upr-1|}>\frac{|r|}{|upr-1|}=|G_pG_{upr}(1)-1|.
\]
Hence $G_pG_{upr}(1)\in H$, from which the assertion follows. We now define $C=C_1\cup C_2\cup\{1\}$, a compact subset of $H$. Then $S_n\to c$ uniformly on $C$.

Let $k(n)$ denote the number of complete blocks in the sequence $e_1,e_2,\dots,e_n$. Then $k(n)\to \infty$ as $n\to\infty$. There is a point $z_n$ equal to $1$, $G_e(1)$, or $G_pG_{upr}(1)$, for some $e\in\mathbb{Z}[i]\setminus\Lambda$ or $p,r\in\mathcal{S}$ and $u\in\{\pm 1, \pm i\}$, with $G_{e_1}G_{e_2}\dotsb G_{e_n}(1)=S_{k(n)}(z_n)$. Since $z_n\in C$, we have that $S_{k(n)}(z_n)\to c$, so $G_{e_1}G_{e_2}\dotsb G_{e_n}(1)\to c$, as required.

We can apply the same proof to show that $\overline{U_1}$ is an exclusion set, this time with $\mathcal{S}=\{3\pm 2i\}$. In this case the hypothesis $\{p,d\}\notin\overline{U_2}$ of Lemma~\ref{lemma52} is automatically satisfied.
\end{proof}

Finally we can prove Theorem~\ref{theoremD}.

\begin{proof}[Proof of Theorem~\ref{theoremD}]
Proposition~\ref{prop52} tells us that $\overline{U_1}$ and $\overline{U_2}$ are reversible exclusion sets over the Gaussian integers.

Suppose now that $U$ is any reversible exclusion set of length two over $\mathbb{Z}[i]$. Consider the eleven periodic continued fractions with periods $(1)$, $(1,2)$, $(1,3)$, or any of the eight words of Figure~\ref{figure6}(a)--(h). Each of these continued fractions corresponds to a closed path in the Gaussian Farey graph, as shown in Figure~\ref{figure2}(a)--(c) and Figure~\ref{figure6}, so all eleven continued fractions diverge (a fact that one can also easily establish algebraically). Now, for each of these continued fractions, the pairs $\{b_{n-1},b_n\}$ for $n\geq 1$ are \emph{all equal}, up to symmetry. Explicitly, the pairs are, in order, $\{1,1\}$, $\{1,2\}$, $\{1,3\}$, $\{1,i\}$, $\{1,1+i\}$, $\{1,2+i\}$, $\{1,3+i\}$, $\{1,2+2i\}$, $\{1+i,1-i\}$, $\{1+i,2\}$, and $\{1+i,2-i\}$. Consequently, for each of these pairs $\{a,b\}$ we must have $a\in U$, $b\in U$, or $\{a,b\}\in U$. Additionally, from Figure~\ref{figure74},  either (i) $1\in U$, $1+2i\in U$, or $\{1,1+2i\}\in U$; or (ii) $1+i\in U$, $2+i\in U$, or $\{1+i,2+i\}\in U$. Hence either $\overline{U_1}\preceq U$ or $\overline{U_2}\preceq U$. Observe that $\overline{U_1}$ and $\overline{U_2}$ themselves are incomparable because $[\overline{1,1+2i}]\in \Ex(\overline{U_1})\setminus \Ex(\overline{U_2})$ and $[\overline{1+i,2+i}]\in \Ex(\overline{U_2})\setminus \Ex(\overline{U_1})$. Hence $\overline{U_1}$ and $\overline{U_2}$ are the only minimal reversible exclusion sets over $\mathbb{Z}[i]$, as required. 
\end{proof}

Theorem~\ref{theoremD} classifies the minimal \emph{reversible} exclusion sets over the Gaussian integers of length two; the problem of classifying the ordered exclusion sets remains open.

\begin{question}
Classify the minimal exclusion sets over the Gaussian integers of length two.
\end{question}

A first step towards answering this question would be to obtain an exclusion set of length two that excludes fewer  continued fractions than the exclusion set of Dani and Nogueira  \cite{DaNo2014}*{Theorem~6.7} specified in the introduction.

Another, separate open problem is to determine exclusion sets for Eisenstein integers.

\begin{question}
Classify the minimal reversible exclusion sets over the Eisenstein integers of length two.
\end{question}

The Farey graph associated to the Eisenstein integers is the 1-skeleton of a tessellation of hyperbolic three-space by tetrahedra (see, for example, \cite{FeKaSeTu2023}). This could conceivably make  exclusion sets over Eisenstein integers more tractable to compute than those over Gaussian integers, for which the associated Farey graph is the 1-skeleton of an octahedral tessellation.

We finish with the remark that with more careful (if lengthy) estimates -- which we do not include here -- one can show that if $[b_0,b_1,\dotsc]$ contains no substrings that belong to $\overline{U_1}$ then the denominators $B_n$ of the convergents of this continued fraction satisfy $|B_{n+1}/B_{n-1}|>1+3/(2n)$, for $n>1$. Similar growth estimates were used by Dani and Nogueira \cite{DaNo2014} to prove density of the values  of certain Gaussian binary quadratic forms. It would be of interest to extend this result to other classes of quadratic forms arising from the optimal exclusion set $\overline{U_1}$.

\begin{figure}[H]
\centering
\MFigSlim{\begin{tikzpicture}[x=1.85cm,y=1.85cm]
  \BBoxOneOne\GridOneOne
  \Edge{(0,1)}{(1,1)}\Edge{(1,1)}{(0.5,0.5)}\Edge{(0.5,0.5)}{(1,0)}\Edge{(1,0)}{(0,0)}
  \BendEdge{(0,0)}{(-0.3,0.5)}\BendEdge{(-0.3,0.5)}{(0,1)}
  \foreach \p in {(0,0),(1,0),(0.5,0.5),(1,1),(0,1),(-0.3,0.5)} {\Spot{\p}}
  \node[rational,below left=2pt] at (0,0) {$0$};
  \node[rational,below right=2pt] at (1,0) {$1$};
  \node[rational,right=3pt] at (0.5,0.5) {$\tfrac{1}{1-i}$};
  \node[rational,above right=2pt] at (1,1) {$1+i$};
  \node[rational,above left=2pt] at (0,1) {$i$};
  \node[rational,left=2pt] at (-0.3,0.5) {$\infty$};
    
  \node[gint,above right=2pt] at (0,0) {$-1$};
  \node[gint,above left=1pt,xshift=-5pt] at (1,0) {$i$};
  \node[gint,below=2pt] at (0.5,0.5) {$1$};
  \node[gint,below left=1pt,xshift=-4pt] at (1,1) {$-i$};
  \node[gint,below right=2pt] at (0,1) {$-1$};
  \node[gint,right=2pt] at (-0.3,0.5) {$i$};
\end{tikzpicture}}{(a) $(i,-1,-i,1,i,-1)^{\pm}$}
\hfill
\MFigSlim{\begin{tikzpicture}[x=1.85cm,y=1.85cm]
  \BBoxOneOne\GridOneOne
  \Edge{(0,0)}{(1,0)}\Edge{(1,0)}{(1,1)}\Edge{(1,1)}{(0.5,0.5)}\Edge{(0.5,0.5)}{(0,0)}
  \foreach \p in {(0,0),(1,0),(1,1),(0.5,0.5)} {\Spot{\p}}
  \node[rational,below left=2pt] at (0,0) {$0$};
  \node[rational,below right=2pt] at (1,0) {$1$};
  \node[rational,above right=2pt] at (1,1) {$1+i$};
  \node[rational,above left=1pt] at (0.5,0.5) {$\tfrac{1}{1-i}$};
  \node[gint,above right=2pt,xshift=2pt,yshift=-1pt] at (0,0) {$-i$};
  \node[gint,above left=2pt] at (1,0) {$-1+i$};
  \node[gint,below left=2pt,xshift=2pt,yshift=-2pt] at (1,1) {$1$};
  \node[gint,right=2pt] at (0.5,0.5) {$1+i$};
\end{tikzpicture}}{(b) $(1,1+i,-i,-1+i)^{\pm}$}
\hfill
\MFigWide{\begin{tikzpicture}[
  x=4cm,y=4cm,
]
  \path[use as bounding box] (-0.90,-0.1) rectangle (0.90,1.07);

  \Edge{(0.5,0.5)}{(0,0)}
  \Edge{(0,0)}{(-0.2,0.4)}
  \Edge{(-0.2,0.4)}{(0,0.5)}
  \Edge{(0,0.5)}{(0.2,0.6)}
  \Edge{(0.2,0.6)}{(0,1)}
  \Edge{(0,1)}{(-0.5,0.5)}
  \Edge{(-0.5,0.5)}{(0,0)}
  \Edge{(0,0)}{(0.2,0.4)}
  \Edge{(0.2,0.4)}{(0,0.5)}
  \Edge{(0,0.5)}{(-0.2,0.6)}
  \Edge{(-0.2,0.6)}{(0,1)}
  \Edge{(0,1)}{(0.5,0.5)}

  \foreach \p in {(0.5,0.5),(0,0),(-0.2,0.4),(0,0.5),(0.2,0.6),(0,1),(-0.5,0.5),(0.2,0.4),(-0.2,0.6)} {\Spot{\p}}

  \node[rational,right=2pt] at (0.5,0.5) {$\tfrac12+\tfrac12i$};
  \node[rational,below=2pt] at (0,0) {$0$};
  \node[rational,above left=2pt,xshift=5pt,yshift=-1pt] at (-0.2,0.4) {$-\tfrac15+\tfrac25i$};
  \node[rational,above=2pt] at (0,0.5) {$\tfrac{i}{2}$};
  \node[rational,below right=2pt,xshift=-2pt,yshift=1pt] at (0.2,0.6) {$\tfrac15+\tfrac35i$};
  \node[rational,above=2pt] at (0,1) {$i$};
  \node[rational,left=2pt] at (-0.5,0.5) {$-\tfrac12+\tfrac12i$};
  \node[rational,above right=2pt,xshift=-2pt,yshift=-1pt] at (0.2,0.4) {$\tfrac15+\tfrac25i$};
  \node[rational,below left=2pt,xshift=5pt,yshift=1pt] at (-0.2,0.6) {$-\tfrac15+\tfrac35i$};

  \node[gint,left=2pt] at (0.5,0.5) {$1$};
  \node[gint,right=3pt] at (0,0) {$2+i$};
  \node[gint,right=2pt,yshift=-2pt] at (-0.2,0.4) {$1$};
  \node[gint,right=5pt,yshift=-0.5pt] at (0,0.5) {$2-i$};
  \node[gint,left=2pt,yshift=2pt] at (0.2,0.6) {$1$};
  \node[gint,right=3pt] at (0,1) {$2+i$};
  \node[gint,right=2pt] at (-0.5,0.5) {$1$};
  \node[gint, left=3pt] at (0,0) {$2-i$};
  \node[gint,left=2pt,yshift=-2pt] at (0.2,0.4) {$1$};
  \node[gint,left=5pt,yshift=-0.5pt] at (0,0.5) {$2+i$};
  \node[gint,right=2pt,yshift=2pt] at (-0.2,0.6) {$1$};
  \node[gint, left=3pt] at (0,1) {$2-i$};
\end{tikzpicture}}{(c) $(1,2+i,1,2-i)^3$}

\vspace{0.4em}

\MFig{\begin{tikzpicture}[x=\ScaleOneOne cm,y=\ScaleOneOne cm]
  \BBoxOneOne\GridOneOne
  \Edge{(0,0)}{(0.5,0)}\Edge{(0.5,0)}{(1,0)}\Edge{(1,0)}{(1.5,0.5)}\Edge{(1.5,0.5)}{(1,1)}\Edge{(1,1)}{(0.5,1)}\Edge{(0.5,1)}{(0,1)}\Edge{(0,1)}{(-0.5,0.5)}\Edge{(-0.5,0.5)}{(0,0)}
  \foreach \p in {(0,0),(-0.5,0.5),(0,1),(0.5,1),(1,1),(1.5,0.5),(1,0),(0.5,0)} {\Spot{\p}}
  \node[rational,below=3pt] at (0,0) {$0$};
  \node[rational,left=3pt] at (-0.5,0.5) {$-\tfrac12+\tfrac12i$};
  \node[rational,above=3pt] at (0,1) {$i$};
  \node[rational,above=3pt] at (0.5,1) {$\tfrac12+i$};
  \node[rational,above=3pt] at (1,1) {$1+i$};
  \node[rational,right=3pt] at (1.5,0.5) {$\tfrac32+\tfrac12i$};
  \node[rational,below=3pt] at (1,0) {$1$};
  \node[rational,below=3pt] at (0.5,0) {$\tfrac{1}{2}$};
  \node[gint,above right=2pt,xshift=-1pt] at (0,0) {$3+i$};
  \node[gint,right=2pt] at (-0.5,0.5) {$1$};
  \node[gint,below right=2pt,yshift=-1pt] at (0,1) {$3-i$};
  \node[gint,below=3pt] at (0.5,1) {$1$};
  \node[gint,below left=2pt,xshift=1pt,yshift=-1pt] at (1,1) {$3+i$};
  \node[gint,left=2pt] at (1.5,0.5) {$1$};
  \node[gint,above left=2pt,xshift=1pt] at (1,0) {$3-i$};
  \node[gint,above=2pt] at (0.5,0) {$1$};
\end{tikzpicture}}{(d) $(1,3+i,1,3-i)^2$}
\quad
\MFig{\begin{tikzpicture}[x=\ScaleOneOne cm,y=\ScaleOneOne cm]
  \BBoxOneOne\GridOneOne
  \Edge{(0,0)}{(0,0.5)}\Edge{(0,0.5)}{(0,1)}\Edge{(0,1)}{(0.5,1)}\Edge{(0.5,1)}{(1,1)}\Edge{(1,1)}{(1,0.5)}\Edge{(1,0.5)}{(1,0)}\Edge{(1,0)}{(0.5,0)}\Edge{(0.5,0)}{(0,0)}
  \foreach \p in {(0,0),(0,0.5),(0,1),(0.5,1),(1,1),(1,0.5),(1,0),(0.5,0)} {\Spot{\p}}
  \node[rational,below left=2pt] at (0,0) {$0$};
  \node[rational,left=3pt] at (0,0.5) {$\tfrac{i}{2}$};
  \node[rational,above left=2pt] at (0,1) {$i$};
  \node[rational,above=3pt] at (0.5,1) {$\tfrac12+i$};
  \node[rational,above right=2pt] at (1,1) {$1+i$};
  \node[rational,right=3pt] at (1,0.5) {$1+\tfrac{i}{2}$};
  \node[rational,below right=2pt] at (1,0) {$1$};
  \node[rational,below=3pt] at (0.5,0) {$\tfrac{1}{2}$};
  \node[gint,above right=2pt] at (0,0) {$2+2i$};
  \node[gint,right=2pt] at (0,0.5) {$-i$};
  \node[gint,below right=2pt,xshift=-1pt,yshift=-1pt] at (0,1) {$-2+2i$};
  \node[gint,below=2pt] at (0.5,1) {$-1$};
  \node[gint,below left=2pt,xshift=1pt,yshift=-1pt] at (1,1) {$-2-2i$};
  \node[gint,left=2pt] at (1,0.5) {$i$};
  \node[gint,above left=2pt] at (1,0) {$2-2i$};
  \node[gint,above=2pt] at (0.5,0) {$1$};
\end{tikzpicture}}{(e) $(1,2+2i,-i,-2+2i)^{\pm}$}
\MFig{\begin{tikzpicture}[x=\ScaleOneOne cm,y=\ScaleOneOne cm]
  \BBoxOneOne\GridOneOne
  \Edge{(0,0)}{(1,0)}\Edge{(1,0)}{(1,1)}\Edge{(1,1)}{(0,1)}\Edge{(0,1)}{(0,0)}
  \foreach \p in {(0,0),(0,1),(1,1),(1,0)} {\Spot{\p}}
  \node[rational,below left=2pt] at (0,0) {$0$};
  \node[rational,above left=2pt] at (0,1) {$i$};
  \node[rational,above right=2pt] at (1,1) {$1+i$};
  \node[rational,below right=2pt] at (1,0) {$1$};
  \node[gint,above right=2pt] at (0,0) {$1+i$};
  \node[gint,below right=2pt] at (0,1) {$1-i$};
  \node[gint,below left=2pt] at (1,1) {$1+i$};
  \node[gint,above left=2pt] at (1,0) {$1-i$};
\end{tikzpicture}}{(f) $(1+i,1-i)^2$}

\vspace{0.40em}

\MFig{\begin{tikzpicture}[x=\ScaleTwoTwo cm,y=\ScaleTwoTwo cm]
  \BBoxTwoTwo\GridTwoTwo
  \Edge{(-1,-1)}{(0,-1)}\Edge{(0,-1)}{(1,-1)}\Edge{(1,-1)}{(1,0)}\Edge{(1,0)}{(1,1)}\Edge{(1,1)}{(0,1)}\Edge{(0,1)}{(-1,1)}\Edge{(-1,1)}{(-1,0)}\Edge{(-1,0)}{(-1,-1)}
  \foreach \p in {(-1,-1),(-1,0),(-1,1),(0,1),(1,1),(1,0),(1,-1),(0,-1)} {\Spot{\p}}
  \node[rational,below left=3pt] at (-1,-1) {$-1-i$};
  \node[rational,left=3pt] at (-1,0) {$-1$};
  \node[rational,above left=3pt] at (-1,1) {$-1+i$};
  \node[rational,above=3pt] at (0,1) {$i$};
  \node[rational,above right=3pt] at (1,1) {$1+i$};
  \node[rational,right=3pt] at (1,0) {$1$};
  \node[rational,below right=3pt] at (1,-1) {$1-i$};
  \node[rational,below=3pt] at (0,-1) {$-i$};
  \node[gint,above right=3pt,xshift=-2pt] at (-1,-1) {$-1-i$};
  \node[gint,below right=3pt] at (-1,0) {$2i$};
  \node[gint,below right=3pt] at (-1,1) {$1-i$};
  \node[gint,below=2pt,xshift=-4pt] at (0,1) {$2$};
  \node[gint,below left=3pt] at (1,1) {$1+i$};
  \node[gint,below left=3pt] at (1,0) {$-2i$};
  \node[gint,above left=3pt] at (1,-1) {$-1+i$};
  \node[gint,above=2pt,xshift=-6pt] at (0,-1) {$-2$};
\end{tikzpicture}}{(g) $(1+i,2,1-i,2i)^{\pm}$}
\qquad
\MFig{\begin{tikzpicture}[x=\ScaleTwoTwo cm,y=\ScaleTwoTwo cm]
  \BBoxTwoTwo\GridTwoTwo
  \Edge{(-1,-1)}{(0,-1)}\Edge{(0,-1)}{(0.5,-0.5)}\Edge{(0.5,-0.5)}{(1,0)}\Edge{(1,0)}{(1,1)}\Edge{(1,1)}{(0,1)}\Edge{(0,1)}{(-0.5,0.5)}\Edge{(-0.5,0.5)}{(-1,0)}\Edge{(-1,0)}{(-1,-1)}
  \foreach \p in {(-1,-1),(-1,0),(-0.5,0.5),(0,1),(1,1),(1,0),(0.5,-0.5),(0,-1)} {\Spot{\p}}
  \node[rational,below left=3pt] at (-1,-1) {$-1-i$};
  \node[rational,left=3pt] at (-1,0) {$-1$};
  \node[rational,above left=1pt,xshift=10pt,yshift=5] at (-0.5,0.5) {$-\tfrac12+\tfrac12i$};
  \node[rational,above=3pt] at (0,1) {$i$};
  \node[rational,above right=3pt] at (1,1) {$1+i$};
  \node[rational,right=3pt] at (1,0) {$1$};
  \node[rational,below right=1pt,xshift=-6pt,yshift=-3] at (0.5,-0.5) {$\tfrac12-\tfrac12i$};
  \node[rational,below=3pt] at (0,-1) {$-i$};
  \node[gint,above right=2pt,xshift=-1pt] at (-1,-1) {$1+i$};
  \node[gint,below right=2pt] at (-1,0) {$1-2i$};
  \node[gint,below right=2pt,xshift=-3pt,yshift=-1pt] at (-0.5,0.5) {$1+i$};
  \node[gint,below right=2pt] at (0,1) {$2-i$};
  \node[gint,below left=2pt] at (1,1) {$1+i$};
  \node[gint,above left=2pt] at (1,0) {$1-2i$};
  \node[gint,above left=2pt] at (0.5,-0.5) {$1+i$};
  \node[gint,above left=2pt,xshift=1pt] at (0,-1) {$2-i$};
\end{tikzpicture}}{(h) $(1+i,2-i,1+i,1-2i)^2$}

\caption{Loops in the Gaussian Farey graph, where $(x)^\pm=(x,-x)$}
\label{figure6}
\end{figure}
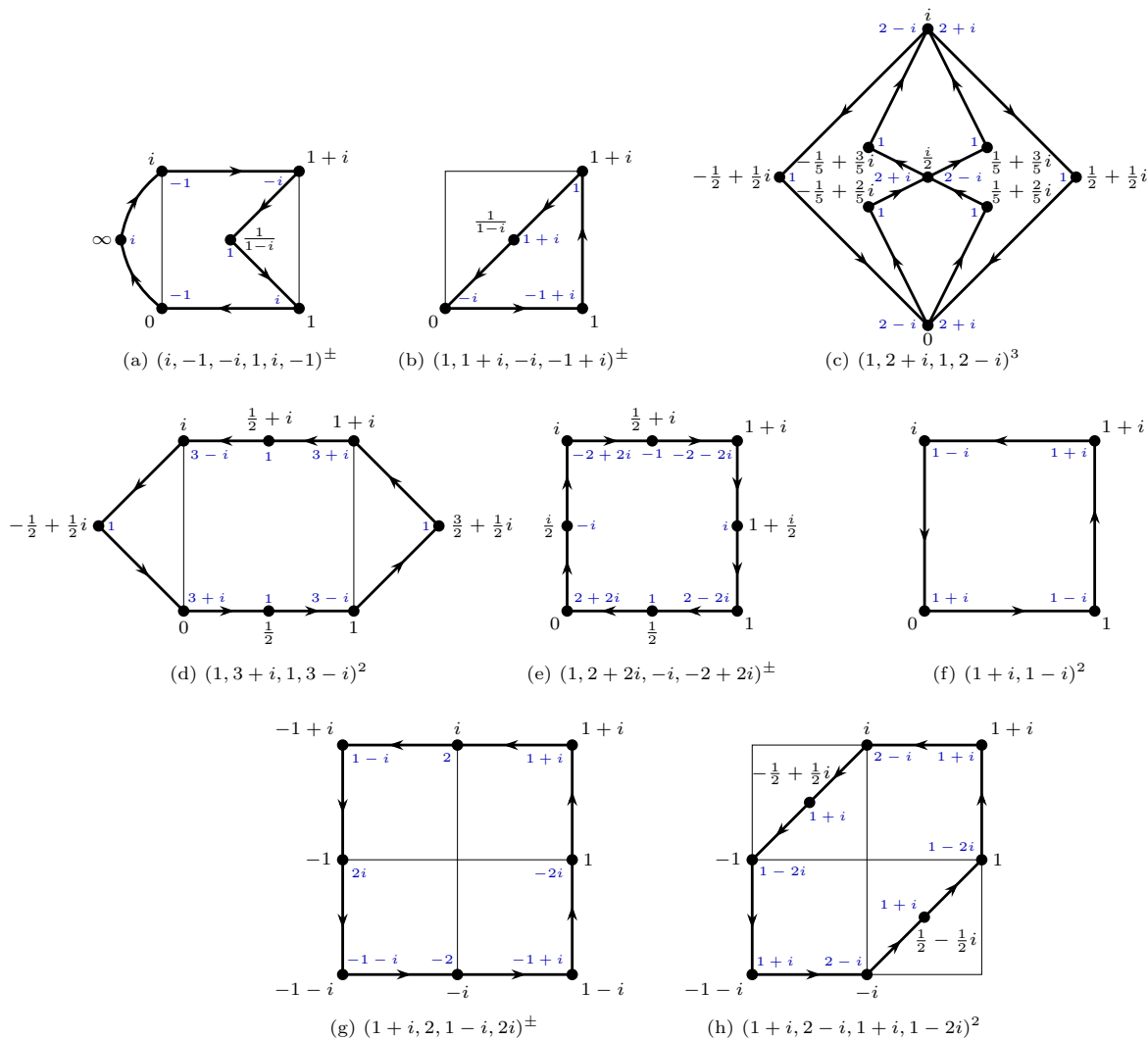

\subsection*{Acknowledgements}
Short and van Son were supported by EPSRC grant EP/W002817/1, and Zabolotskii was supported by EPSRC DTP grant EP/W524098/1. The authors acknowledge the workshops \emph{Frieze patterns in algebra, combinatorics and geometry} (CIRM, May 2025) and \emph{Farey's legacy in frieze patterns and discrete geometry} (ICMS, April 2026) at which work from this paper was shared. 

\subsection*{AI declaration}
This paper was developed over two years by the authors for project EP/W002817/1. The final version benefitted from interactions with ChatGPT for creating TikZ figures, correcting errors, and streamlining proofs, notably those of Lemma~\ref{lemma70} and Theorem~\ref{theoremD}.

\subsection*{Data availability} 
There is no data associated with this article.

\subsection*{Conflict of interest} 
The authors have no conflict of interest.


\begin{bibdiv}
\begin{biblist}

\bib{AgSo2006}{article}{
   author={Aguiar, Marcelo},
   author={Sottile, Frank},
   title={\href{https://doi.org/10.1016/j.jalgebra.2005.06.021}{Structure of the Loday-Ronco Hopf algebra of trees}},
   journal={J. Algebra},
   volume={295},
   date={2006},
   number={2},
   pages={473--511}, 
}
%
\bib{Be2003}{article}{
   author={Beardon, Alan F.},
   title={\href{https://doi.org/10.1023/A:1022116901444}{The Hillam--Thron theorem in higher dimensions}},
   journal={Geom. Dedicata},
   volume={96},
   date={2003},
   pages={205--209},
}
%
\bib{BeHoSh2012}{article}{
   author={Beardon, A. F.},
   author={Hockman, M.},
   author={Short, I.},
   title={\href{https://doi.org/10.1307/mmj/1331222851}{Geodesic continued fractions}},
   journal={Michigan Math. J.},
   volume={61},
   date={2012},
   number={1},
   pages={133--150},
}
%
\bib{CoCo1973}{article}{
   author={Conway, J. H.},
   author={Coxeter, H. S. M.},
  title={\href{https://doi.org/10.2307/3615344}{Triangulated polygons and friezes}},
   journal={Math. Gaz.},
   volume={57},
   date={1973},
   pages={87--94, 175--183},
}
%
\bib{Cu2018}{article}{
   author={Cuntz, M.},
   title={\href{https://doi.org/10.1016/j.jalgebra.2018.01.028}{On subsequences of quiddity cycles and Nichols algebras}},
   journal={J. Algebra},
   volume={502},
   date={2018},
   pages={315--327},
}
%
\bib{DaNo2014}{article}{
   author={Dani, S. G.},
   author={Nogueira, Arnaldo},
   title={\href{https://doi.org/10.1090/S0002-9947-2014-06003-0}{Continued fractions for complex numbers and values of binary
   quadratic forms}},
   journal={Trans. Amer. Math. Soc.},
   volume={366},
   date={2014},
   number={7},
   pages={3553--3583},
}
%
\bib{FeKaSeTu2023}{article}{
   author={Felikson, Anna},
   author={Karpenkov, Oleg},
   author={Serhiyenko, Khrystyna},
   author={Tumarkin, Pavel},
   title={\href{https://doi.org/10.1007/s10711-025-00997-5}{$3d$ Farey graph, lambda lengths and $SL_2$-tilings}},
   journal={Geom. Dedicata},
   volume={219},
   date={2025},
   number={2},
   pages={Paper No. 33},
}
%
%
\bib{Ha2022}{book}{
   author={Hatcher, Allen},
   title={Topology of numbers},
   publisher={American Mathematical Society, Providence, RI},
   date={2022},
   pages={ix+341},
   isbn={978-1-4704-5611-5},
}
%
\bib{Ho2020}{article}{
   author={Hockman, Meira},
   title={\href{https://doi.org/10.1307/mmj/1576033219}{Geodesic Gaussian integer continued fractions}},
   journal={Michigan Math. J.},
   volume={69},
   date={2020},
   number={2},
   pages={297--322},
}
%
\bib{KaUg2005}{article}{
   author={Katok, Svetlana},
   author={Ugarcovici, Ilie},
   title={\href{https://www.mathnet.ru/eng/mmj307}{Geometrically Markov geodesics on the modular surface}},
   journal={Mosc. Math. J.},
   volume={5},
   date={2005},
   number={1},
   pages={135--155},
}
%
\bib{LeTh1942}{article}{
   author={Leighton, Walter},
   author={Thron, W. J.},
   title={\href{https://doi.org/10.1215/S0012-7094-42-00952-9}{Continued fractions with complex elements}},
   journal={Duke Math. J.},
   volume={9},
   date={1942},
   pages={763--772},
}
%
\bib{LoWa2008}{book}{
   author={Lorentzen, Lisa},
   author={Waadeland, Haakon},
   title={Continued fractions. Vol. 1},
   series={Atlantis Studies in Mathematics for Engineering and Science},
   volume={1},
   edition={2},
   note={Convergence theory},
   publisher={Atlantis Press, Paris; World Scientific Publishing Co. Pte.
   Ltd., Hackensack, NJ},
   date={2008},
   pages={xii+308},
}
%
\bib{ScWa1940}{article}{
   author={Scott, W. T.},
   author={Wall, H. S.},
   title={\href{https://doi.org/10.1090/S0002-9947-1940-0001320-1}{A convergence theorem for continued fractions}},
   journal={Trans. Amer. Math. Soc.},
   volume={47},
   date={1940},
   pages={155--172},
}
%
\bib{ShSt2022}{article}{
   author={Short, Ian},
   author={Stanier, Margaret},
   title={\href{https://doi.org/10.1090/proc/15574}{Necessary and sufficient conditions for convergence of integer
   continued fractions}},
   journal={Proc. Amer. Math. Soc.},
   volume={150},
   date={2022},
   number={2},
   pages={617--631},
}
%
\bib{Ti1911}{article}{
   author={Tietze, Heinrich},
   title={\href{https://doi.org/10.1007/BF01461159}{\"Uber Kriterien f\"ur Konvergenz und Irrationalit\"at unendlicher Kettenbr\"uche}},
   journal={Math. Ann.},
   volume={70},
   date={1911},
   number={2},
   pages={236--265},
}


\end{biblist}
\end{bibdiv}
\end{document}